\documentclass[journal,twoside]{IEEEtran}

\usepackage{amsmath}		
\usepackage{amssymb}		
\usepackage{amsfonts}		
\usepackage{amsthm}		    

\usepackage{mathtools}		
\usepackage[export]{adjustbox}

\usepackage{dsfont}		

\usepackage[dvipsnames,svgnames]{xcolor}		
\colorlet{MyBlue}{DodgerBlue!75!Black}
\colorlet{MyGreen}{DarkGreen!85!Black}

\usepackage{array}		
\usepackage{booktabs}		
\usepackage[inline,shortlabels]{enumitem}		
\usepackage{multirow}
\usepackage{threeparttable}   

\usepackage[utf8]{inputenc}
\usepackage{tikz}
\usepackage{pgfplots}
\usepgfplotslibrary{groupplots}
\usetikzlibrary{pgfplots.groupplots}
\usetikzlibrary{matrix}
\usepackage{acronym}		
\usepackage{latexsym}		
\usepackage{xfrac}		
\usepackage{xspace}		
\usepackage{xparse}     

\usepackage{hyperref}
\hypersetup{
colorlinks=true,
linktocpage=true,
pdfstartview=FitH,
breaklinks=true,
pdfpagemode=UseNone,
pageanchor=true,
pdfpagemode=UseOutlines,
plainpages=false,
bookmarksnumbered,
bookmarksopen=false,
bookmarksopenlevel=1,
hypertexnames=true,
pdfhighlight=/O,
urlcolor=Maroon,linkcolor=MyBlue!60!black,citecolor=MyBlue!70!black,	
pdftitle={},
pdfauthor={},
pdfsubject={},
pdfkeywords={},
pdfcreator={pdfLaTeX},
pdfproducer={LaTeX with hyperref}
}
\usepackage[sort&compress,capitalize,nameinlink]{cleveref}		
\crefname{assumption}{Assumption}{Assumptions}
\crefname{assumptionloc}{Assumption}{Assumptions}

\renewcommand{\thesubsubsection}{\arabic{subsubsection}}

\usepackage{algorithm}		
\usepackage{algpseudocode}		

\usepackage{thmtools,thm-restate,thm-autoref}

\theoremstyle{plain}
\newtheorem{theorem}{Theorem}		
\newtheorem{lemma}{Lemma}		
\newtheorem{proposition}{Proposition}		

\newtheorem*{corollary*}{Corollary}		

\theoremstyle{definition}
\newtheorem{assumption}{Assumption}		

\newtheorem*{definition*}{Definition}		
\newtheorem*{assumption*}{Assumptions}		

\makeatletter		
\newcommand{\asmtag}[1]{
  \let\oldtheassumption\theassumption
  \renewcommand{\theassumption}{#1}
  \g@addto@macro\endassumption{
    \addtocounter{assumption}{-1}
    \global\let\theassumption\oldtheassumption}
  }
\makeatother

\theoremstyle{remark}
\newtheorem*{remark*}{Remark}		

\usepackage{macros}

\usepackage[showdeletions]{color-edits}		
\addauthor[Yu-Guan]{YGH}{ForestGreen}

\addauthor[Franck]{FI}{Orange}

\begin{document}
%
\title{Push--Pull with Device Sampling}
%
%
%

\author{Yu-Guan Hsieh, Yassine Laguel, Franck Iutzeler, Jérôme Malick
\thanks{Yu-Guan Hsieh and Franck Iutzeler are with Université Grenoble Alpes, Grenoble, France (email: yu-guan.hsieh@univ-grenoble.alpes; franck.iutzeler@univ-grenoble.alpes).}
\thanks{Yassine Laguel was with Université Grenoble Alpes, Grenoble, France. He is now with Rutgers University, NJ 07102 USA (email: laguel.yassine@gmail.com).}
\thanks{Jérôme Malick is with CNRS and Université Grenoble Alpes, Grenoble, France
(email: jerome.malick@cnrs.fr).}}

%
%

\markboth{}%
{Hsieh \MakeLowercase{\textit{et al.}}: Push--Pull with Device Sampling}
%

\maketitle

\begin{abstract}
We consider decentralized optimization problems in which a number of agents collaborate to minimize the average of their local functions by exchanging over an underlying communication graph. 
Specifically, we place ourselves in an asynchronous model where only a random portion of nodes perform computation at each iteration, while the information exchange can be conducted between all the nodes and in an asymmetric fashion.
For this setting, we propose an algorithm that combines gradient tracking with a network-level variance reduction (in contrast to variance reduction within each node).
This enables the nodes to track the average of the gradients of the objective functions.
Our theoretical analysis shows that the algorithm converges linearly, when the local objective functions are strongly convex, under mild connectivity conditions on the expected mixing matrices. 
In particular, our result does not require the mixing matrices to be doubly stochastic.
In the experiments, we investigate a broadcast mechanism that transmits information from computing nodes to their neighbors, and confirm the linear convergence of our method on both synthetic and real-world datasets.

\end{abstract}


\section{Introduction}
\label{set:intro}
\IEEEPARstart{I}{n} this paper, we focus on solving the optimization problem
\begin{equation}
    \tag{P}
    \label{eq:problem}
    \min_{\point\in\vecspace} \obj(\point) \defeq
    \frac{1}{\nWorkers} \sumworker \vw[\obj](\point)
\end{equation}
where each function $\vw[\obj]\from\vecspace\to\R$ is available only locally at the $\worker$-th node of a graph. Hence, in order to reach a consensus on the minimum of \eqref{eq:problem}, the $\nWorkers$ nodes have to communicate using the graph's edges.

Such decentralized optimization problems have been widely studied in the literature at least since the pioneering works of Bertsekas and Tsitsiklis \cite{tsitsiklis1986distributed,bertsekas2015parallel}. In terms of applications, decentralized optimization methods are popular for regression or classification problems when the communication possibilities between the nodes are scarce and cannot be handled by a central entity (\eg for wireless sensor networks,  IoT-enabled edge devices, etc.); see the recent surveys \cite{nedic2018network,yang2019survey,assran2020advances,nedic2020distributed}. In these applications, the workload and communications between the nodes are of primary importance.

The computation at the node level mainly depends on which optimization method serves as a basis. 
If the nodes are able to solve optimization sub-problems, the Alternating Direction Method of Multipliers (ADMM) and other dual methods can be extended to distributed setting \cite{duchi2011dual,boyd2011distributed}. At the other end of the spectrum, stochastic gradients methods are very popular since they require minimal computation at each node \cite{nedic2009distributed}. 
Gradient-based methods offer a good comprise between these two extremes and currently know a rebirth, especially for machine learning applications; see \eg the recent \cite{kovalev2020optimal}. 

In terms of exchanges, all communications between nodes have to go through the edges of the graph. If the graph is undirected (\ie the edges are all bidirectional), the nodes can gossip to average their values. Mathematically, this corresponds to multiplying the agents' states by a doubly-stochastic matrix; see \cite[Sec. II]{nedic2020distributed} for details. 
However, if the graph is directed, such direct gossiping is no longer possible since maintaining both a consensus among the nodes and the average of their values is not possible at the same time \cite{hendrickx2015fundamental}. To overcome this problem, two main type of methods have been developed. First, Push--Sum methods (or ratio consensus) consist in exchanging an additional ``weighting''; these methods can reach an average consensus for the ratio of the two values \cite{kempe2003gossip,ICH13,nedic2016stochastic}. However, the analysis of Push--Sum gradient methods is often quite involved and the algorithm can become numerically unstable due to division by very small values, see \eg the simulations of \cite{PSXN20} as well as references therein. Second, Push--Pull methods rely on two communications steps with different mixings to maintain convergence, offering strong theoretical guarantees as well as good practical performance \cite{XK18,PSXN20}.

Finally, a desirable feature in decentralized methods is the possibility to allow the nodes to randomly awaken, compute, and send/receive information; which is generally called randomized gossiping \cite{boyd2006randomized} or asynchronous decentralized methods \cite{assran2020advances,iutzeler2013asynchronous}. In terms of analysis, this consists in replacing the fixed communication matrices with random ones having the support corresponding to the active links; this was actively studied for decentralized gradient methods, including Push--Pull gradient \cite{PSXN20}.

\subsection{Contributions and outline}

In this paper, we focus on gradient-based methods for decentralized asynchronous optimization on directed graphs. 
We propose and analyze an asynchronous Push--Pull gradient algorithm where only a fraction of the nodes are actively computing a local gradient at each iteration. This feature is inspired from the device sampling (or client selection) procedure in federated learning \cite{MMRH+17,KMAB+21}. This popular mechanism enables to take into account the fact that all the nodes may not be available at all time and furthermore that querying all gradients at each iteration may be a waste of computational power if the nodes' values only change by a little amount.

In terms of algorithm, device sampling calls for a variance reduction mechanism in order to mitigate the noise induced by the sampling of the nodes. 
To achieve this, we introduce a SAGA-like \cite{DBL14} update at the network level; see Example\;\ref{ex:saga}.
This additional step thus requires an original analysis. 

The remainder of the paper is structured as follows.
The introduction is completed by an overview on related literature.
In \cref{sec:algo}, we present our general algorithmic template (\acl{PPDS}), together with some specific cases of interest, connecting our method with existing methods. 
In \cref{sec:conv}, we provide linear convergence results under classical convexity/smoothness assumptions on the objective functions and weak assumptions on communications. 
\cref{sec:anal} and \cref{sec:sim} are dedicated to the detailed convergence analysis and illustrative numerical simulations.
Finally, proofs of a couple of   technical intermediate results are given in Appendix.

\subsection{Related works}
Direct extensions of the gradient method to the decentralized setting rely on decreasing stepsizes to converge and are thus limited to sublinear convergence rates, even if the minimized functions are smooth and strongly convex. To overcome this situation, the gradient tracking technique was introduced; it consists in dynamically tracking the average value of the gradient and using this value instead of the local gradient. This technique enables the use of a fixed stepsize and exhibits much better rates in theory and in practice  \cite{xu2015augmented,nedic2017achieving,qu2017harnessing,shi2015extra}; see also the recent \cite{li2020revisiting}. 
Gradient tracking can be intuitively seen as a variance reduction at the network level. The method presented in this paper extends this idea of variance reduction to device sampling. 
Note also that in the case where the nodes' functions are themselves a finite sum, this sum can be sampled, and variance reduction can be additionally applied at the node level \cite{XKK20,hendrikx2021optimal}, however this specific form is out of the scope of the present paper.

AB/Push--Pull gradient methods naturally involve gradient tracking; see \eg \cite[Rem.~1]{PSXN20} and more generally \cite{XK18,ZY19,SXK20,PSXN20}. These algorithms share common ingredients and mainly differ in their communications models. The works that are the most closely related to the asynchronous directed setup considered in this paper are \cite{PSXN20} and \cite{ZY19}. These two papers study an asynchronous version of AB/Push--Pull which share similarities, in the update and the communication scheme, with our proposed method (more precisely, with the special setup of Example~\ref{ex:broadcast}).
However, in contrast to our method, these methods require every node that is involved in the communication step to perform a local update.
Moreover, the analysis of \cite{PSXN20} only works for the more restrictive case where the non-diagonal coefficients of the mixing matrices are sufficiently small.
Note finally that \cite{ZY19} does not consider a random network model, but instead performs an analysis in terms of the worst-case dependence on the delays. This analysis is thus complementary to our work.


\subsection{Basic notation and definitions}

Throughout the paper, we use bold lowercase letters to denote vectors and capital letters to denote matrices.
$\Id_k$ and $\ones_k$ respectively represent the identity matrix of size $k\times k$ and the $k-$dimensional vector containing all ones.
The subscript is omitted when the dimension is clear from the context.
We also define $\avgMat=\ones_{\nWorkers}\ones_{\nWorkers}^{\top}/\nWorkers$ as the projection matrix onto the consensus space,
and denote by $\sradius(P)$ the spectral radius of a matrix $P$.

The interaction topology between the nodes is modeled by a directed graph $\graph=(\vertices,\edges)$, where $\vertices$ is the set of vertices (nodes) and $\edges\in\vertices\times\vertices$ is the set of edges, such that node $\worker$ can send information to $\workeralt$ only if $(\worker,\workeralt)\in\edges$.
The out-neighbors and in-neighbors of a node $\worker$ are respectively defined by
\begin{equation}
    \notag
    \nhdgraph^{out}_{\worker}=\setdef{\workeralt\in\workers}{(\worker,\workeralt)\in\edges},
    ~\nhdgraph^{in}_{\worker}=\setdef{\workeralt\in\workers}{(\workeralt,\worker)\in\edges}.
\end{equation}
When the graph is undirected, the two sets coincide and we simply write $\nhdgraph_{\worker}$.
We say that the 
matrix $\mat=(\matEle_{\worker\workeralt})\in\R^{\nWorkers\times\nWorkers}$ is compatible with the underlying communication topology if $\matEle_{\worker\workeralt}=0$ whenever $(\workeralt,\worker)\notin\edges$.

Finally we introduce the aggregate objective function, $\objMulti(\jstate)=\sumworker\vw[\obj](\vw[\state])$, as a function of the 
variable $\jstate=[\vw[\state][1],\ldots,\vw[\state][\nWorkers]]^{\top}\!\in\R^{\nWorkers\times\vdim}$. 
When $\objMulti$ is differentiable, we have
\begin{equation}
    \notag
    \grad\objMulti(\jstate)=[\grad\vw[\obj][\start](\vw[\state][1]),
    \ldots, \grad\vw[\obj][\nWorkers](\vw[\state][\nWorkers])]^{\top}.
\end{equation}


\section{Algorithms: existing, new, and examples}
\label{sec:algo}


In this section, we present our asynchronous Push-Pull gradient algorithm with device sampling.
Prior to that, we recall the existing AB/Push--Pull method \cite{XK18,PSXN20} which inspires our algorithm.
After detailing our general template, we instantiate it in several situations of interest, revealing its versatility.




\subsection{The AB/Push--Pull method}\label{sec:PP}

If all the nodes are active at each iteration, 
the communication setup reduces to that of synchronous decentralized optimization.
In this situation and assuming that the functions $f_i$ are differentiable, the AB/Push--Pull algorithm \cite{XK18,PSXN20}
is described as follows.
In addition to the decision variables $\vwt[\state]$ that should minimize~$\obj$, a variable $\vwt[\gstate]$ is introduced to track the gradient of $\obj$. 
Then, provided  $\step>0$ a constant stepsize and two mixing matrices $\mixing=(\vm[\mixingEle])\in\R^{\nWorkers\times\nWorkers}$ and $\mixingalt=(\vm[\mixingaltEle])\in\R^{\nWorkers\times\nWorkers}$
, the update of the algorithm at iteration\;$\run$ writes
\begin{equation}
    \notag
    \begin{aligned}
    \vwtupdate[\state] 
    &=\sum_{\workeralt\in\workers}\vm[\mixingEle]\vwt[\state][\workeralt] - \step\,\vwt[\gstate],\\
    \vwtupdate[\gstate]
    &=\sum_{\workeralt\in\workers}\vm[\mixingaltEle]\vwt[\gstate][\workeralt]
    +\grad\vw[\obj](\vwtupdate[\state])-\grad\vw[\obj](\vwt[\state]).
    \end{aligned}
\end{equation}
%
It is required that $\mixing$ and $\mixingalt$ have non-negative weights and 
be respectively row-stochastic ($\mixing\ones=\ones$) and column-stochastic ($\ones^{\top}\mixingalt=\ones^{\top}$).
With the notation $\vt[\jgstate]=[\vwt[\gstate][1],\ldots,\vwt[\gstate][\nWorkers]]^{\top}$,
the update can also be written, in a matrix form, as
\begin{equation}
    \label{eq:PP}
    \tag{PP}
    \begin{aligned}
    \update[\jstate] &= \vt[\mixing]\vt[\jstate] - \step \vt[\jgstate],\\
    \update[\jgstate] &= \vt[\mixingalt]\vt[\jgstate] + \grad\objMulti(\update[\jstate])-\grad\objMulti(\vt[\jstate]).
    \end{aligned}
\end{equation}

Intuitively, the use of row-stochastic matrices drive $\vwt[\state]$ to consensus, while the use of column-stochastic matrices preserves the total mass, \ie $\ones^{\top}\mixingalt\vvec=\ones^{\top}\vvec$ for any $\vvec\in\R^{\nWorkers}$.
Moreover, if the difference $\grad\vw[\obj](\vwtupdate[\state])-\grad\vw[\obj](\vwt[\state])$ tends to zero, $\vwt[\gstate]$ converges to a multiple of $\sumworker\grad\vw[\obj](\vwt[\state])$.
In fact, from the Perron-Frobenius theorem, we know that if $\mixingalt$ is primitive\footnote{A square non-negative matrix $\mat$ is called primitive if there exists a power $k\ge1$ such that $\mat^{k}>0$; see \cite[Th.~8.5.2]{horn2012matrix}} then $\lim_{\toinf}\mixingalt^{\run}=\rightEigvecB\ones^{\top}$ where $\rightEigvecB$ is the right eigenvector of $\mixingalt$ associated with the eigenvalue $1$ such that $\ones^{\top}\rightEigvecB=1$.
Therefore, asymptotically every $\vwt[\state]$ descends in the direction opposite to the gradient of $\obj$.
Mathematically, it can be proven that under standard convexity assumptions AB/Push--Pull converges linearly with sufficiently small constant step-size $\step$ \cite[Th.~1]{pu2018push}.

\subsection{Proposed Push--Pull with Device Sampling}

Our algorithm can be viewed as a generalization of  AB/Push--Pull to handle the device sampling mechanism.
First, in order to allow for asynchronicity, let $\seqinf[\vt[\mixing]]$ and $\seqinf[\vt[\mixingalt]]$ be two sequences of mixing matrices that are compatible with the underlying communication topology.
Now, to handle device sampling, we denote by $\vt[\workers]$ the set of nodes that are active at time $\run$.
This means that node $\worker$ computes a local gradient at round $\run$ if and only if $\worker\in\vt[\workers]$.

With the notations $\vt[\mixing]=(\vmt[\mixingEle])$, $\vt[\mixingalt]=(\vmt[\mixingaltEle])$, and $\vt[\sampMat]=\diag(\one_{\worker\in\vt[\workers]})$, \ie $\vt[\sampMat]$ is the diagonal matrix in $\R^{\nWorkers\times\nWorkers}$ whose $\worker$-th diagonal element is $1$ if $\worker\in\vt[\workers]$ and $0$ otherwise, each iteration of our proposed \ac{PPDS} can be stated in the compact form
\begin{equation}
    \label{eq:PPDS}
    \tag{PPDS}
    \begin{aligned}
    \inter[\jgstate] &= \vt[\jgstate] + \vt[\sampMat](\grad\objMulti(\vt[\jstate])-\grad\objMulti(\vt[\jcstate])),\\
    \inter[\jstate] &= \vt[\jstate] - \step \vt[\sampMat]\inter[\jgstate],\\
    \update[\jcstate] &= \vt[\sampMat] \vt[\jstate] + (\Id-\vt[\sampMat]) \vt[\jcstate],\\
    \update[\jgstate] &= \vt[\mixingalt]\inter[\jgstate], ~~~
    \update[\jstate] = \vt[\mixing]\inter[\jstate].
    \end{aligned}
\end{equation}

Several remarks are in order.
First, we introduce auxiliary local variable $\vwt[\cstate]$ for each node and write $\vt[\jcstate]=[\vwt[\cstate][1],\ldots,\vwt[\cstate][\nWorkers]]^{\top}$. 
The presence of these variables indicate the nodes store their last computed gradient. This is necessary because $\vwt[\state]$ can be modified by network communication between two successive activations of node $\worker$.
In fact, while only the active nodes perform local updates at each iteration, the inactive nodes can be involved in the communication process.
This flexibility allows us to take into account a wider class of algorithms, as illustrated in the forthcoming examples.
Finally, as in AB/Push--Pull, we only require the matrices $\seqinf[\vt[\mixing]]$ and $\seqinf[\vt[\mixingalt]]$ to be respectively row- and column-stochastic.
This means that we allow for one-way communication and in particular inactive nodes may passively receive information without sending back their local states. 

In terms of implementation, our method \eqref{eq:PPDS} gives \cref{algo:PPDS} for asynchronous optimization on directed graphs. In the next section, we discuss special cases and show that we recover existing algorithms.

\begin{algorithm}[!ht]
    \caption{\acs{PPDS} (at each node $\worker$)}
    \label{algo:PPDS}
\begin{algorithmic}[1]
    \State {\bfseries Initialize:}
         $\vwt[\gstate][\worker][\start]=\grad\vw[\obj](\vwt[\state][\worker][\start]$);
         $\vwt[\cstate][\worker][\start]=\vwt[\state][\worker][\start]$
    \For{$\run=\running$}
    \State \texttt{\underline{Local update}}
    \vskip 0.2em
    \If{$\worker\in\vt[\workers]$}
    \vskip 0.1em
    \State
    $\vwtinter[\gstate] \subs \vwt[\gstate]+\grad\vw[\obj](\vwt[\state])-\grad\vw[\obj](\vwt[\cstate])$
    \State $\vwtinter[\state] \subs \vwt[\state] - \step\vwtinter[\gstate]$
    \State Set $\vwtupdate[\cstate] \subs \vwt[\state]$ and store $\grad\vw[\obj](\vwtupdate[\cstate])$
    \Else
    \State $\vwtinter[\gstate] \subs \vwt[\gstate]$;~ $\vwtinter[\state] \subs \vwt[\state]$;~ $\vwtupdate[\cstate] \subs \vwt[\cstate]$
    \EndIf
    \State \texttt{\underline{Communication}}
    \vskip 0.2em
    \State $\vwtupdate[\state] =\sum_{\workeralt\in\workers}\vmt[\mixingEle]\vwtinter[\state][\workeralt]$
    \Comment{$\vt[\mixing]$ is row-stochastic}
    \State $\vwtupdate[\gstate] =\sum_{\workeralt\in\workers}\vmt[\mixingaltEle]\vwtinter[\gstate][\workeralt]$
    \Comment{$\vt[\mixingalt]$ is column-stochastic}
    \EndFor
\end{algorithmic}
\end{algorithm}

\subsection{Special cases}\label{sec:special}

\subsubsection{AB/Push--Pull}

We first demonstrate that the original AB/Push-Pull algorithm \cite{XK18,PSXN20} indeed falls within the \ac{PPDS} framework.
For this, we fix $\vt[\mixing]\equiv\mixing$, $\vt[\mixingalt]\equiv\mixingalt$, and $\vt[\workers]=\workers$.
Then, after rearranging, the \eqref{eq:PPDS} update can be written as

\begin{equation}
    \notag
    \begin{aligned}
    \update[\jstate] &= \mixing(\vt[\jstate]-\step\inter[\jgstate]),\\
    \vt[\jgstate][\run+\frac{3}{2}]
    &= \mixingalt\inter[\jgstate] + \grad\objMulti(\update[\jstate])
    -\grad\objMulti(\vt[\jstate]).
    \end{aligned}
\end{equation}
This is exactly the adapt-then-combine variant of AB/Push--Pull as presented in \cite{PSXN20}.

\vskip 2pt
\subsubsection{Communication between active nodes}
\label{ex:active-com}
We can imagine a situation where only active agents participate in the communication.
Then, these active agents may communicate with each other using mixing matrices $\mixing(\vt[\workers]), \mixingalt(\vt[\workers])$ that are compatible with the induced subgraph $\inducedgraph{\vt[\workers]}$, defined by the vertex set\;$\vt[\workers]$ and the edges of $\edges$ that connect two vertices of $\vt[\workers]$. For example, if the graph is symmetric and $\mixing(\vt[\workers])=\mixingalt(\vt[\workers])$ is the Metropolis matrix of $\inducedgraph{\vt[\workers]}$, we have $\vt[\mixing]=\vt[\mixingalt]$ and
\begin{equation}
    \notag
    \vmt[\mixingEle]
    = \left\{
        \begin{array}{ll}
            \frac{1}{\max(
            \degr_{\run}(\worker),
            \degr_{\run}(\workeralt))} & \text{if } \worker\neq\workeralt \text{ and }
            \{\worker,\workeralt\}\in\edges\intersect 2^{\vt[\workers]},\\
            1 - \sum_{\indg=1}^{\nWorkers}
            \vmt[\mixingEle][\worker\indg]
            \one_{\indg\neq\workeralt}
            & \text{if } \worker=\workeralt,\\
            0 & \text{otherwise},
        \end{array}
        \right.
\end{equation}
where $\degr_{\run}(\worker)=\card(\vw[\nhdgraph]\intersect\vt[\workers])$ is the degree of $\worker\in\vt[\workers]$ in the induced graph $\inducedgraph{\vt[\workers]}$.

\vskip 2pt
\subsubsection{Broadcast-type update}
\label{ex:broadcast}
As mentioned previously, our algorithm allows inactive nodes to passively receive information from active nodes.
Therefore, the active nodes can simply broadcast their local variables to their neighbors, no matter whether these neighbors are active or not.
To ensure the row-stochasticity of $\vt[\mixing]$, the received $\vwt[\state]$'s are averaged out.
On the other hand, to guarantee the column-stochasticity of $\vt[\mixingalt]$, an active node divides its $\vwt[\gstate]$ by the number of nodes it sends the information to, 
as usually done in a push-sum scheme. 

For concreteness, let us denote by $\nhdgraph^{out}_{\workeralt,\run}$  the set of neighbors that active worker $\workeralt\in\vt[\workers]$ transmit information to (including itself) in round $\run$ and by $\nhdgraph^{in}_{\worker,\run}=\setdef{\workeralt\in\vt[\workers]}{\worker\in\nhdgraph^{out}_{\workeralt,\run}}$ the set of active workers that send information to $\worker$ in this same round.
The mixing matrices $\vt[\mixing]$ and $\vt[\mixingalt]$ are then defined by
\begin{align*}
    \vmt[\mixingEle] &= 
    \left\{
    \begin{array}{ll}
        \makebox[2.6cm][l]{$\frac{1}{\card(\nhdgraph^{in}_{\worker,\run}\union\{\worker\})}$} &\text{ if $\workeralt\in\nhdgraph^{in}_{\worker,\run}\union\{\worker\}$},\\
        0 &\text{ otherwise;}\\
    \end{array}
    \right.
\end{align*}
\begin{align*}    \vmt[\mixingaltEle] &= 
    \left\{
    \begin{array}{ll}
        \makebox[2.6cm][l]{$\frac{1}{\card(\nhdgraph^{out}_{\workeralt,\run})}$} & \text{ if $\workeralt\in\nhdgraph^{in}_{\worker,\run}$},\\
        1 & \text{ if $\workeralt=\worker$ and $\workeralt\notin\vt[\workers]$,} \\
        0 & \text{ otherwise.}\\
    \end{array}
    \right.
\end{align*}

In this example, we see that our method offers an additional degree of freedom compared to \GPP \cite{PSXN20} since in that algorithm $\vmt[\mixingEle]>0$ only if $\worker,\workeralt \in \vt[\workers]$; this is not necessarily the case in our approach.

\vskip 2pt
\subsubsection{SAGA}
\label{ex:saga}
SAGA \cite{DBL14} is a well-known (centralized) variance reduction methods that replaces the stochastic gradient $\grad\vw[\obj](\vt[\state])$ with an unbiased gradient estimator with diminishing variance.
For this, we store a table of gradients $(\grad\vw[\obj](\vwt[\cstate]))_{\worker=1}^{\nWorkers}$, where, similar to \ac{PPDS}, $\vwt[\cstate]$ is the iterate at which $\grad\vw[\obj]$ was last evaluated.
Let $\worker_{\run}$ be sampled from the index set $\{1,\ldots,\nWorkers\}$.
The update of SAGA is then given by:
\begin{align}
    \label{eq:SAGA-vrgvec}
    \vt[\vrgvec] &= \grad\vw[\obj][\worker_{\run}](\vt[\state]) -  \grad\vw[\obj][\worker_{\run}](\vt[\cstate]) + \frac{1}{\nWorkers} \sumworker \grad\vw[\obj][\worker](\vt[\cstate]),\\
    \update[\state] &=
    \vt[\state] - \step\vt[\vrgvec].
    \notag
\end{align}
%

To recover SAGA from \ac{PPDS}, we set $\vt[\mixing]=\vt[\mixingalt]\equiv\avgMat$.
This ensures $\vwt[\gstate]=(1/\nWorkers)\sumworker \grad\vw[\obj][\worker](\vt[\cstate])$ and thus $\vwtinter[\gstate]$ is exactly updated as in \eqref{eq:SAGA-vrgvec} when $\worker$ is active. 
Specifically, if exactly one node is sampled at each iteration, \eqref{eq:PPDS} with step-size $\step$ and $\vt[\mixing]=\vt[\mixingalt]\equiv\avgMat$ is equivalent to SAGA with stepsize $\step/\nWorkers$.
If multiple workers are active at a same time slot, we get a mini-batch version of SAGA.

\section{Linear convergence of \acs{PPDS}}
\label{sec:conv}

In this section, we present convergence guarantees of \ac{PPDS} for strongly convex functions over a random network model.
Concretely,
we make the following standard convexity/smoothness assumption on the objective functions:

\begin{assumption}
\label{asm:functions}
All the individual $\vw[\obj]$'s are $\lips$-smooth and convex; the global function $\obj$ is $\str$-strongly convex.
\end{assumption}

Thanks to the strong convexity of $\obj$, we know there exists a unique solution of \eqref{eq:problem} which we will denote by $\sol$.
%
Moreover, we model $\seqinf[\vt[\workers]]$, $\seqinf[\vt[\mixing]]$, and $\seqinf[\vt[\mixingalt]]$ as random variables satisfying that:

\begin{assumption}
\label{asm:independence}
The random variables $\seqinf[(\vt[\workers], \vt[\mixing], \vt[\mixingalt])]$ are temporally \ac{iid}.
\end{assumption}

\cref{asm:independence} is actually only needed to provide the contractions of \cref{lem:mixing-contraction,lem:mixing-contraction-multi}. Hence,  it could be weakened accordingly. We chose to keep it as such for ease of reading and for consistency with the literature.

\subsection{The general case}

First, we present our linear convergence result under rather weak assumptions on communications (essentially that the information can flow all over the network) and device sampling (each node is sampled with positive probability). 

\begin{assumption}
\label{asm:matrices}
The mixing matrices $\seqinf[\vt[\mixing]]$ and $\seqinf[\vt[\mixingalt]]$ have the following properties:
\begin{enumerate}[a\upshape)] 
\item For all $\run\in\N$, $\vt[\mixing]$ is row-stochastic and $\vt[\mixingalt]$ is column-stochastic.
\item Both $\mixing\defeq\ex[\vt[\mixing][\start]]$ and $\mixingalt\defeq\ex[\vt[\mixingalt][\start]]$ are primitive.
\item There exists $\diaglow>0$ such that $\vmt[\mixingEle][\worker\worker]\ge\diaglow$ and $\vmt[\mixingaltEle][\worker\worker]\ge\diaglow$ for all $\worker\in\workers,\run\in\N$.
\label{asm-matrices-c}
\end{enumerate}
\end{assumption}

\begin{assumption}
\label{asm:sampling}
Every node is sampled with positive probability, \ie 
$\vw[\sampprob]\defeq\prob(\worker\in\vt[\workers][\start])>0$ for all $\worker\in\workers$. 
\end{assumption}

It is straightforward to verify that \cref{asm:independence,asm:matrices,asm:sampling} are fulfilled in all the aforementioned examples.
In particular, in Example \ref{ex:broadcast}, the primitivity of matrices $\mixing$ and $\mixingalt$ are ensured by the strong connectivity of the underlying graph $\graph$ since $\vm[\mixing]>0$ if and only if $(\workeralt,\worker)\in\edges$ if and only if $\vm[\mixingalt]>0$.
On the other hand, \cref{asm:matrices}c posits that at each iteration each nodes maintains a fraction of its previous states.
This rules out the counterexamples in which the states of the active nodes are always overwritten by those of the inactive states.
Under these fairly weak assumptions, we manage to prove the convergence of \ac{PPDS} as stated in the following theorem.

\begin{theorem}
\label{thm:cvg}
Let \cref{asm:functions,asm:independence,asm:matrices,asm:sampling} hold.
If \eqref{eq:PPDS} is run with a sufficiently small step-size $\step>0$, then
\begin{enumerate}[a\upshape)]
    \item $\vwt[\state]$ converges almost surely to the solution $\sol$.
    \item The expected squared distance between the iterate and the solution $\ex[\norm{\vwt[\state]-\sol}^2]$ vanishes geometrically.
\end{enumerate}
\end{theorem}

\cref{thm:cvg} shows that the nice properties of gradient tracking and variance reduced methods are also preserved by our algorithm: it converges with constant step-size and enjoys a linear convergence rate as centralized gradient descent.
Therefore, our method effectively reduces the variances of the noises induced by both sampling and communication.

We note that the assumptions for this result are quite similar to the ones for \GPP in \cite{PSXN20}, except that i) we do not put additional restrictions on the coefficients of the gossip matrix (unlike Eq.~(24a) in \cite{PSXN20}); and ii) we allow for a device sampling strategy that can be independent or correlated with the gossiping step.

\subsection{The case of doubly stochastic matrices.}
Due to the generality of the result, \cref{thm:cvg} only describes the qualitative behavior of the algorithm.
To derive 
a convergence rate with explicit dependence to the problem parameters, we focus on the specific situation where the mixing matrices are doubly stochastic and the active devices are sampled uniformly at random.
Formally, we make the following assumptions.

\asmtag{\ref*{asm:matrices}$'$}
\begin{assumption}
\label{asm:matrices-bi}
For all $\run\in\N$, both $\vt[\mixing]$ and $\vt[\mixingalt]$ are doubly stochastic.
Moreover, we have the inequality 
\begin{equation}
    \notag
    \cfactor\defeq \max
    (\sradius(\ex[\vt[\mixing^{\top}][\start](\Id-\avgMat)\vt[\mixing][\start]]),
    \sradius(\ex[\vt[\mixingalt^{\top}][\start](\Id-\avgMat)\vt[\mixingalt][\start]]))<1.
\end{equation}
%
%
\end{assumption}

\asmtag{\ref*{asm:sampling}$'$}
\begin{assumption}
\label{asm:sampling-uni}
$\vt[\workers]$ is of fixed size $\nWorkersActive$ and is sampled uniformly from all the subsets of this size.
\end{assumption}

The bistochasticity assumption is for example verified when the communication matrices are the Metropolis matrices of the subgraphs of the active nodes (covered by Example\;\ref{ex:active-com} of \cref{sec:special}).
However, it is still possible to have communications between active and inactive nodes under this assumption, as demonstrated by the SAGA example (example \ref{ex:saga}).
On the technical side, the bistochasticity of the matrices and the condition $\cfactor<1$ allows us to derive a per-step contraction of the variance of the nodes' variables (see \cref{lem:mixing-contraction}).
Using uniform sampling further facilitates the analysis and makes the final expression much more concise.
Building upon these,
the next theorem states the step-size condition and the convergence rate of \ac{PPDS} when these assumptions are fulfilled.

\begin{theorem}
\label{thm:cvg-bistoch}
Let \cref{asm:functions,asm:independence,asm:matrices-bi,asm:sampling-uni} hold.
If \eqref{eq:PPDS} is run with step-size
\begin{equation}
    \label{eq:eta-condition}\resizebox{0.9\linewidth}{!}{$%
    \step\le\min\left(
    \frac{(1-\cfactor)^2}{14\lips}\sqrt{\frac{\nWorkers}{\nWorkersActive}},
    \frac{(1-\cfactor)^2}{2304\lips}\left(\frac{\nWorkers}{\nWorkersActive}\right)^{\frac{3}{2}},
    \frac{1}{576\lips}\sqrt{\frac{\nWorkers}{\nWorkersActive}}
    \right),$}
\end{equation}
then the expected squared distance 
$\ex[\norm{\vwt[\state]-\sol}^2]$ vanishes geometrically in $\bigoh(\cvgrate^{\run})$ with $\cvgrate=\max\left(1-\frac{\step\str\nWorkersActive}{2\nWorkers},1-\frac{\nWorkersActive}{4\nWorkers}\right)$.
In particular, it takes 
\begin{equation}
    \label{eq:n-iterations}
    \bigoh\left(\left(\frac{\lips}{\str}\sqrt{\frac{\nWorkers}{\nWorkersActive}}\frac{1}{(1-\cfactor)^2}+\frac{\nWorkers}{\nWorkersActive}\right)\log\left(\frac{1}{\eps}\right)\right)
\end{equation}
iterations to achieve $\eps$ accuracy when $\step$ is suitably tuned.
\end{theorem}

\cref{thm:cvg-bistoch} indicates that a larger step-size can (and should) be taken for smaller sample size, if all the other parameters are fixed.
Intuitively, this is because at each iteration fewer gradients enter the network, and thus these gradients can be used with a larger weight.

In term of dependence on problem parameters, the linear dependence on the condition number $\lips/\str$ matches that of standard gradient descent, whereas the $1/(1-\cfactor)^2$ dependence on mixing parameter is common in the literature of gradient tracking \cite{qu2017harnessing} but has been further improved recently by \cite{koloskova2021improved} in the case of single fixed mixing matrix.
Regarding the effect of device sampling, although the complexity in terms of iterations is degraded by $\sqrt{\nWorkers/\nWorkersActive}$ compared to asynchronous Push--Pull without device sampling (\ie $\nWorkersActive=\nWorkers$), the complexity in number of computed gradients is actually improved.
To see this, we multiply \eqref{eq:n-iterations} by $\nWorkersActive$ and and verify that the resulting quantity indeed decreases when $\nWorkersActive$ gets smaller.

Nonetheless, device sampling may also affect the connectivity of the network and thus $\cfactor$ if the communication matrices are chosen according to the sampled devices $\vt[\workers]$ (for instance, in Example~\ref{ex:broadcast}).
Therefore, unlike in the centralized case, sampling with variance reduction is not always guaranteed to converge faster here. Rather, there is a communication-computation trade-off
that involves both the choice of the sampling size $\nWorkersActive$ and the mixing matrices $\vt[\mixing]$, $\vt[\mixingalt]$ (see \cref{algo:PPDS} for details).



\section{Convergence analysis}
\label{sec:anal}
In this section we outline the proofs of \cref{thm:cvg,thm:cvg-bistoch}.
To begin, let us define
\begin{equation}
    \notag
    \vwt[\vrgvec]=\vwt[\gstate]+\grad\vw[\obj](\vwt[\state])-\grad\vw[\obj](\vwt[\cstate]).
\end{equation}
as the gradient estimator of node $\worker$ at iteration $\run$ so that $\vwtinter[\gstate]=\vwt[\vrgvec]$ if and only if $\worker\in\vt[\workers]$, and $\vwtinter[\gstate]=\vwt[\gstate]$ otherwise.
With the mass preservation property of column-stochastic matrices and the definition of $\vwt[\cstate]$, we have immediately the following lemma.

\begin{lemma}
\label{lem:yt-gt-sum}
Suppose that the matrices $\seqinf[\vt[\mixingalt]]$ are column-stochastic. It holds that
\begin{equation}
    \notag
    \sumworker\vwt[\gstate]=\sumworker\grad\vw[\obj](\vwt[\cstate]), ~~
    \sumworker\vwt[\vrgvec]=\sumworker\grad\vw[\obj](\vwt[\state]).
\end{equation}
\end{lemma}

Therefore, if the iterates move in the direction $-\sumworker\vwt[\vrgvec]$, we can expect convergence of the algorithm.
This idea is crucial for our proof.

Another important step in the analysis is to establish that the nodes' decisions variables converge to a consensus.
For this, let us write $\vt[\avg[\state]] = \ones^{\top}\vt[\jstate]/\nWorkers$ for the average of these variables.
Similarly, we also use the notation $\vt[\avg[\gstate]] = \ones^{\top}\vt[\jgstate]/\nWorkers$.

Finally, we would like to highlight that the expectation $\ex$ is taken over the randomness induced by both sampling and communication.
We define $\seqinf[\vt[\filter]]$ as the natural filtration associated to the sequence $\seqinf[\vt[\jstate]]$ so that $((\vt[\workers][\runalt], \vt[\mixing][\runalt], \vt[\mixingalt][\runalt]))_{\start\le\runalt\le\run-1}$ is $\vt[\filter]$ measurable while $(\vt[\workers],\vt[\mixing],\vt[\mixingalt])$ is not.
For simplicity, we write $\exoft$ for the expectation conditioned on the history up to time $\run$,
\ie $\exoft[\cdot]=\exof{\cdot\given{\vt[\filter]}}=\exof{\cdot\given{((\vt[\workers][\runalt], \vt[\mixing][\runalt], \vt[\mixingalt][\runalt]))_{\start\le\runalt\le\run-1}}}$.

\subsection{Analysis with doubly stochastic matrices}
\label{subsec:anal-bistoch}

As a warm-up, we first establish the convergence of the algorithm in the simpler case where both \cref{asm:matrices-bi} and \cref{asm:sampling-uni} hold.
This allows us to highlight our proof strategy without having to deal with the additional difficulties caused by the fact of having asymmetric communications.
Following previous works that analyze gradient tracking and variance reduced methods, the essential idea of our proof is to derive a system of inequalities for the following quantities

\begin{equation}
    \label{eq:rec-variable-bistoch}
    \begin{aligned}
    \vt[\distsol] &= \ex[\norm{\vt[\avg[\state]]-\sol}^2],~~
    \vt[\errgap] = \ex[\obj(\vt[\avg[\state]]) - \obj(\sol)],~~\\
    \vt[\varX] &= \ex[\norm{\vt[\jstate]-\avgconcat}^2],~~
    \vt[\varY] = \ex[\norm{\vt[\jgstate]-\avgconcat[\gstate]}^2],\\
    \vt[\cmeasure] &= \sumworker
    \ex[\norm{\grad\vw[\obj](\vwt[\cstate])-\grad\vw[\obj](\sol)}^2].
    \end{aligned}
\end{equation}
Here, $\vt[\distsol]$ and $\vt[\errgap]$ measure the performance of the averaged iterate; $\vt[\varX]$ and $\vt[\varY]$ measure the variances of the two variables of the agents; and $\vt[\cmeasure]$ measures the quality of the control variates and is standard in the analysis of variance reduced algorithms \cite{JZ13,KM19}.
The following proposition bounds these quantities by a linear combination of their previous values.

\begin{proposition}
\label{prop:recineq-matrix}
Let $\vt[\recvec] = [\vt[\distsol], \vt[\varX], \vt[\varY], \vt[\cmeasure]]^{\top}$.
Under \cref{asm:functions,asm:independence,asm:matrices-bi,asm:sampling-uni}, we have
\begin{equation}
    \label{eq:recineq-matrix}
    \update[\recvec]\le\recmat\vt[\recvec]+\vt[\errgap]\recbias
\end{equation}
where the entries of $\recmat$ and $\recbias$ are given by
\begin{equation}
\notag
\renewcommand\arraystretch{1.8}
    \recmat =
    \begin{bmatrix}
    1-\frac{\step\str\nWorkersActive}{2\nWorkers} & 
    \frac{\step\lips\nWorkersActive}{\nWorkers^2}
    +\frac{10\step^2\lips^2\nWorkersActive^2}{\nWorkers^3} &
    \frac{2\step^2\nWorkersActive^2}{\nWorkers^3} &
    \frac{4\step^2\nWorkersActive^2}{\nWorkers^3}
    \\
    0 & \frac{1+\cfactor}{2} + \frac{20\step^2\lips^2\nWorkersActive}{\nWorkers(1-\cfactor)} &
    \frac{4\step^2\nWorkersActive}{\nWorkers(1-\cfactor)} &
    \frac{8\step^2\nWorkersActive}{\nWorkers(1-\cfactor)}
    \\
    0 & \frac{8\lips^2\nWorkersActive}{\nWorkers(1-\cfactor)} & 
    \frac{1+\cfactor}{2} &
    \frac{4\nWorkersActive}{\nWorkers(1-\cfactor)}
    \\
    0 & \frac{2\lips^2\nWorkersActive}{\nWorkers} & 0 &
    1-\frac{\nWorkersActive}{\nWorkers}
    \end{bmatrix},
\end{equation}
\begin{equation}
    \notag
    \recbias=
    \left[
    -\frac{\step\nWorkersActive}{\nWorkers}
    +\frac{20\step^2\lips\nWorkersActive^2}{\nWorkers^2},
    \frac{40\step^2\lips\nWorkersActive}{1-\cfactor},
    \frac{16\lips\nWorkersActive}{1-\cfactor},
    4\lips\nWorkersActive\right]^{\top}.
\end{equation}
\end{proposition}

To prove \cref{prop:recineq-matrix}, we start by presenting a series of technical lemmas that are useful for this purpose.
First, in order to deal with device sampling, we observe that $\vt[\jvrgvec]=[\vwt[\vrgvec][1],\ldots,\vwt[\vrgvec][\nWorkers]]^{\top}$ plays an important role since $\vt[\sampMat]\inter[\jgstate]=\vt[\sampMat]\vt[\jvrgvec]$.
With the uniform sampling of \cref{asm:sampling-uni}, we obtain the following lemma.

\begin{lemma}
\label{lem:yt-gt}
Let \cref{asm:independence,asm:sampling-uni} hold. Then
\begin{flalign*}
    ~~a)~ & \exoft[\ones^{\top}\vt[\sampMat]\inter[\jgstate]]
    =\frac{\nWorkersActive}{\nWorkers}\sumworker\vwt[\vrgvec].&\\
    ~~b)~ & \exoft[\norm{\ones^{\top}\vt[\sampMat]\inter[\jgstate]}^2]
    \le\frac{\nWorkersActive^2}{\nWorkers}\sumworker\norm{\vwt[\vrgvec]}^2.&
\end{flalign*}
\end{lemma}
\begin{proof}
\textit{a)}
Note that $\ones^{\top}\vt[\sampMat]\inter[\jgstate]
    =\sum_{\worker\in\vt[\workers]}\vwt[\gstate]
    =\sum_{\worker\in\vt[\workers]}\vwt[\vrgvec]$.
Therefore, 
    \begin{equation}
        \notag
        \begin{aligned}
        \exoft[\ones^{\top}\vt[\sampMat]\inter[\jgstate]]
        &=\exoft\left[\sum_{\worker\in\vt[\workers]}\vwt[\vrgvec]\right]
        =\exoft\left[\sumworker \one_{\worker\in\vt[\workers]}\vwt[\vrgvec]\right]\\
        &=\sumworker\vwt[\vrgvec]\exoft[\one_{\worker\in\vt[\workers]}]
        =\frac{\nWorkersActive}{\nWorkers}\sumworker\vwt[\vrgvec].
        \end{aligned}
    \end{equation}
We can put $\vwt[\gvec]$ outside the expectation since it is $\vt[\filter]$-measurable.

\vskip 3pt
\textit{b)}
Similarly, we have
\begin{align*}
    \exoft[\norm{\ones^{\top}\vt[\sampMat]\inter[\jgstate]}^2]
    &=\exoft\left[\left\|\sum_{\worker\in\vt[\workers]}\vwt[\vrgvec]\right\|^2\right]
    \le\exoft\left[\nWorkersActive\sum_{\worker\in\vt[\workers]}\norm{\vwt[\vrgvec]}^2\right]\\
    &=\nWorkersActive\sumworker\exoft[\one_{\worker\in\vt[\workers]}\norm{\vwt[\vrgvec]}^2]
    =\frac{\nWorkersActive^2}{\nWorkers}\sumworker\norm{\vwt[\vrgvec]}^2.
\end{align*}
\end{proof}

To control the distance to consensus, we use the lemma below that shows a contraction property of the mixing matrices.

\begin{lemma}
\label{lem:mixing-contraction}
Let \cref{asm:independence,asm:matrices-bi} hold. Then
\begin{flalign*}
    ~~a)~ & \exoft[\norm{\vt[\mixing]\vt[\jstate]-\avgconcat}^2]
    \le \cfactor \norm{\vt[\jstate]-\avgconcat}^2.&\\
    ~~b)~ & \exoft[\norm{\vt[\mixingalt]\vt[\jgstate]-\avgconcat[\gstate]}^2]
    \le \cfactor \norm{\vt[\jgstate]-\avgconcat[\gstate]}^2.&
\end{flalign*}
\end{lemma}
\begin{proof}
Since $\vt[\mixing]$ is doubly stochastic, we can write
\begin{equation}
    \notag
    \begin{aligned}
        &\exoft[\norm{\vt[\mixing]\vt[\jstate]-\avgconcat}^2]\\
        &~=\exoft[\norm{(\Id-\avgMat)\vt[\mixing](\vt[\jstate]-\avgconcat)}^2]\\
        &~=\exoft[\tr[(\vt[\jstate]-\avgconcat)^{\top}\vt[\mixing]^\top(\Id-\avgMat)^2\vt[\mixing](\vt[\jstate]-\avgconcat)]]\\
        &~=\tr[(\vt[\jstate]-\avgconcat)^{\top}\exoft[\vt[\mixing]^\top(\Id-\avgMat)\vt[\mixing]](\vt[\jstate]-\avgconcat)]\\
        &~\le\sradius(\exoft[\vt[\mixing]^\top(\Id-\avgMat)\vt[\mixing]])\norm{\vt[\jstate]-\avgconcat}^2.
    \end{aligned}
\end{equation}
Under \cref{asm:independence} we have $\exoft[\vt[\mixing]^\top(\Id-\avgMat)\vt[\mixing]]=\ex[\vt[\mixing^{\top}][\start](\Id-\avgMat)\vt[\mixing][\start]]$
and \textit{a)} follows immediately given that $\sradius(\ex[\vt[\mixing^{\top}][\start](\Id-\avgMat)\vt[\mixing][\start]])\le\cfactor$.
Property \textit{b)} is proved in the same way.
\end{proof}

Finally, we can use the smoothness of the objective functions and the optimality conditions to bound the expected squared norm of $\vt[\jvrgvec]$ and gradients differences by the quantities introduced in \eqref{eq:rec-variable-bistoch}.

\begin{lemma}
\label{lem:aux}
Let \cref{asm:functions} hold and $\seqinf[\vt[\mixingalt]]$ be column-stochastic. We have
\begin{flalign*}
    ~~a)~ & \ex[\norm{\grad\objMulti(\vt[\jstate])-\grad\objMulti(\ones^\top \sol)}^2]
    \le 2\lips^2\vt[\varX]+4\nWorkers\lips\vt[\errgap].&\\
    ~~b)~ & \ex[\norm{\grad\objMulti(\vt[\jstate])-\grad\objMulti(\vt[\jcstate])}^2]
    \le 4\lips^2\vt[\varX]+8\nWorkers\lips\vt[\errgap]+2\vt[\cmeasure].&\\
    ~~c)~ & \ex[\norm{\vt[\jvrgvec]}^2]
    \le 10\lips^2\vt[\varX]+20\nWorkers\lips\vt[\errgap]+4\vt[\cmeasure]+2\vt[\varY].&
\end{flalign*}
\end{lemma}
\begin{proof}
See \cref{app:proof-aux}.
\end{proof}

We are now ready to prove \cref{prop:recineq-matrix} by leveraging the above lemmas.

\begin{proof}[Proof of \cref{prop:recineq-matrix}]
Below we bound the four quantities in question respectively.

\vskip 3pt
\textbf{Bounding $\update[\distsol]$.}
We develop
\begin{align}
    \notag
    \norm{\update[\avg[\state]]-\sol}^2
    &= \norm{\vt[\avg[\state]]
        -\frac{\step}{\nWorkers}\ones^{\top}\vt[\sampMat]\inter[\jgstate]-\sol}^2 \\
    \notag
    &= \norm{\vt[\avg[\state]]-\sol}^2
    - \frac{2\step}{\nWorkers} 
    \product{\vt[\avg[\state]]-\sol}{\ones^{\top}\vt[\sampMat]\inter[\jgstate]}\\
    & ~ + \frac{\step^2}{\nWorkers^2}\norm{\ones^{\top}\vt[\sampMat]\inter[\jgstate]}^2.
    \label{eq:dt-develop}
\end{align}
Using \cref{lem:yt-gt-sum,lem:yt-gt} and \cref{asm:functions}, we get
\begin{small}
\begin{equation}
\label{eq:dt-salar-product}
\begin{aligned}[b]
&\exoft[\product{\vt[\avg[\state]]-\sol}{\ones^{\top}\vt[\sampMat]\inter[\jgstate]}]\\
&~=\product{\vt[\avg[\state]]-\sol}
{\frac{\nWorkersActive}{\nWorkers}\sumworker\grad\vw[\obj](\vwt[\state])} \\
&~= \frac{\nWorkersActive}{\nWorkers}
\sumworker
(\product{\vt[\avg[\state]]-\vwt[\state]}{\grad\vw[\obj](\vwt[\state])}
+\product{\vwt[\state]-\sol}{\grad\vw[\obj](\vwt[\state])}) \\
&~\ge \frac{\nWorkersActive}{\nWorkers}
\sumworker
(\vw[\obj](\vt[\avg[\state]])-\vw[\obj](\vwt[\state])
-\frac{\lips}{2}\norm{\vwt[\state]-\vt[\avg[\state]]}^2
+ \vw[\obj](\vwt[\state])-\vw[\obj](\sol)) \\
&~=\nWorkersActive(\obj(\vt[\avg[\state]])-\obj(\sol))
-\frac{\lips\nWorkersActive}{2\nWorkers}\norm{\vt[\jstate]-\avgconcat}^2 \\
&~\ge \frac{\nWorkersActive}{2}(\obj(\vt[\avg[\state]])-\obj(\sol))
+\frac{\str\nWorkersActive}{4}\norm{\vt[\avg[\state]]-\sol}^2
-\frac{\lips\nWorkersActive}{2\nWorkers}\norm{\vt[\jstate]-\avgconcat}^2.
\end{aligned}
\end{equation}
\end{small}
In the last line we have used the fact that $\obj(\point)-\obj(\sol)\ge(\str/2)\norm{\point-\sol}^2$ for every $\point\in\vecspace$ since $\obj$ is strongly convex.
As for the last term of \eqref{eq:dt-develop}, we resort to \cref{lem:yt-gt}b and \cref{lem:aux}c. This gives
\begin{equation}
    \notag
    \ex[\norm{\ones^{\top}\vt[\sampMat]\inter[\jgstate]}^2]
    \le \frac{\nWorkersActive^2}{\nWorkers}
    (10\lips^2\vt[\varX]+20\nWorkers\lips\vt[\errgap]+4\vt[\cmeasure]+2\vt[\varY]).
\end{equation}
Combining the above inequalities we get
\begin{equation}
    \notag
    \begin{aligned}
    \update[\distsol]
    &\le
    \left(1-\frac{\step\str\nWorkersActive}{2\nWorkers}\right)\vt[\distsol]
    + \left(\frac{\step\lips\nWorkersActive}{\nWorkers^2}
    + \frac{10\step^2\lips^2\nWorkersActive^2}{\nWorkers^3}\right)\vt[\varX]\\
    &~+ \frac{2\step^2\nWorkersActive^2}{\nWorkers^3}\vt[\varY]
    + \frac{4\step^2\nWorkersActive^2}{\nWorkers^3}\vt[\cmeasure]
    - \left(\frac{\step\nWorkersActive}{\nWorkers}
    - \frac{20\step^2\lips\nWorkersActive^2}{\nWorkers^2}\right)\vt[\errgap].
    \end{aligned}
\end{equation}

\vskip 3pt
\textbf{Bounding $\update[\varX]$.}
In the inequality $\norm{a+b}^2\le(1+\youngdelta)\norm{a}^2+(1+1/\youngdelta)\norm{b}^2$, choosing $\youngdelta=(1-\cfactor)/2\cfactor$ gives\footnote{Without loss of generality we assume $\cfactor>0$. Otherwise the first term in the inequalities are always $0$ and we can simply take $\youngdelta=1$. The same remark applies to the analysis in \cref{subsec:anal-gen}.}
\begin{equation}
    \label{eq:Young-mixing}
    \norm{a+b}^2\le\frac{1+\cfactor}{2\cfactor}\norm{a}^2+\frac{1+\cfactor}{1-\cfactor}\norm{b}^2.
\end{equation}

Since $\vt[\mixing]$ is doubly stochastic and hence column-stochastic, it holds $\avgMat\vt[\mixing]=\avgMat$. We then have,
\begin{align}
    \notag
    &\exoft[\norm{\update[\jstate]-\avgconcat[\state][{\run+1}]}^2]\\
    \notag
    &~= \exoft[\norm{\vt[\mixing]\vt[\jstate]
    - \step\vt[\mixing]\vt[\sampMat]\inter[\jgstate]
    -(\avgconcat-\step\avgMat\vt[\sampMat]\inter[\jgstate])}^2]\\
    \notag
    &~\le
    \frac{1+\cfactor}{2\cfactor}
    \exoft[\norm{\vt[\mixing]\vt[\jstate]-\avgconcat}^2]\\
    &~~+\frac{1+\cfactor}{1-\cfactor}\step^2
    \exoft[\norm{\vt[\mixing]\vt[\sampMat]\inter[\jgstate]
    -\avgMat\vt[\sampMat]\inter[\jgstate]}^2].
    \label{eq:varX-develop}
\end{align}
Using \cref{lem:mixing-contraction}a the first term can be bounded by $(1+\cfactor)\norm{\vt[\jstate]-\avgconcat}^2/2$.
The same does not apply to the second term as $\vt[\mixing]$ and $\vt[\sampMat]$ are not independent. Nonetheless, with the bistochasticity of $\vt[\mixing]$, we can still write
\begin{align*}
    \exoft[\norm{\vt[\mixing]\vt[\sampMat]\inter[\jgstate]
    -\avgMat\vt[\sampMat]\inter[\jgstate]}^2]
    &\le
    \exoft[\norm{\vt[\sampMat]\inter[\jgstate]}^2]\\
    &=\frac{\nWorkersActive}{\nWorkers}\sumworker\norm{\vwt[\vrgvec]}^2.
\end{align*}
With \cref{lem:aux}c, taking total expectation in \eqref{eq:varX-develop} then gives
\begin{small}
\begin{equation}
    \notag
    \begin{aligned}
    \update[\varX]
    &\le
    \frac{1+\cfactor}{2}\vt[\varX]
    + \frac{1+\cfactor}{1-\cfactor}\frac{\step^2\nWorkersActive}{\nWorkers}
    (10\lips^2\vt[\varX]+20\nWorkers\lips\vt[\errgap]+4\vt[\cmeasure]+2\vt[\varY])
    \\
    &\le
    \left(\frac{1+\cfactor}{2}
    + \frac{20\step^2\lips^2\nWorkersActive}{\nWorkers(1-\cfactor)}\right)\vt[\varX]
    + \frac{4\step^2\nWorkersActive}{\nWorkers(1-\cfactor)}\vt[\varY]\\
    &~+ \frac{8\step^2\nWorkersActive}{\nWorkers(1-\cfactor)}\vt[\cmeasure]
    + \frac{40\step^2\lips\nWorkersActive}{1-\cfactor}\vt[\errgap].
    \end{aligned}
\end{equation}
\end{small}

\vskip 3pt
\textbf{Bounding $\update[\varY]$.}
Similar to the above, using \eqref{eq:Young-mixing} and the bistochasticity of $\vt[\mixingalt]$, we obtain
\begin{small}
\begin{align}
    \notag
    \norm{\update[\jgstate]-\avgconcat[\gstate][\run+1]}^2
    &= \norm{\vt[\mixingalt]\vt[\jgstate]
    -\vt[\mixingalt]\vt[\sampMat](\grad\objMulti(\vt[\jstate])-\grad\objMulti(\vt[\jcstate]))\\
    \notag
    &~~~~-(\avgconcat[\gstate]-
    \avgMat\vt[\sampMat](\grad\objMulti(\vt[\jstate])-\grad\objMulti(\vt[\jcstate])))
    }^2\\
    \notag
    &\le
    \frac{1+\cfactor}{2\cfactor}
    \norm{\vt[\mixingalt]\vt[\jgstate]-\avgconcat[\gstate]}^2\\
    &~+\frac{1+\cfactor}{1-\cfactor}
    \norm{\vt[\sampMat](\grad\objMulti(\vt[\jstate])-\grad\objMulti(\vt[\jcstate]))}^2.
    \label{eq:varY-develop}
\end{align}
\end{small}
%
%
The uniform sampling assumption implies that
\begin{equation}
    \notag
    \begin{aligned}
    &\exoft[\norm{\vt[\sampMat](\grad\objMulti(\vt[\jstate])-\grad\objMulti(\vt[\jcstate]))}^2]\\
    &~= \exoft\left[\sum_{\worker\in\vt[\workers]}
    \norm{\grad\vw[\obj](\vwt[\state])-\grad\vw[\obj](\vwt[\cstate])}^2\right]\\
    &~= \frac{\nWorkersActive}{\nWorkers}\sumworker
    \norm{\grad\vw[\obj](\vwt[\state])-\grad\vw[\obj](\vwt[\cstate])}^2.
    \end{aligned}
\end{equation}
Taking expectation in \eqref{eq:varY-develop} and applying \cref{lem:mixing-contraction}b and \cref{lem:aux}b then yields
\begin{equation}
    \notag
    \begin{aligned}
    \update[\varY]
    &\le
    \frac{1+\cfactor}{2}\vt[\varY]
    + \frac{1+\cfactor}{1-\cfactor}\frac{\nWorkersActive}{\nWorkers}
    (4\lips^2\vt[\varX]+8\nWorkers\lips\vt[\errgap]+2\vt[\cmeasure])\\
    &\le
    \frac{1+\cfactor}{2}\vt[\varY]
    + \frac{8\lips^2\nWorkersActive}{\nWorkers(1-\cfactor)} \vt[\varX]
    + \frac{4\nWorkersActive}{\nWorkers(1-\cfactor)}\vt[\cmeasure]
    + \frac{16\lips\nWorkersActive}{1-\cfactor}\vt[\errgap].
    \end{aligned}
\end{equation}

\vskip 3pt
\textbf{Bounding $\update[\cmeasure]$.}
Note that by the update rule of $\vwt[\cstate]$, we have
\begin{equation}
    \notag
    \begin{aligned}
    &\exoft[\norm{\grad\vw[\obj](\vwtupdate[\cstate])-\grad\vw[\obj](\sol)}^2]\\
    &~= \left(1-\frac{\nWorkersActive}{\nWorkers}\right)
    \norm{\grad\vw[\obj](\vwt[\cstate])-\grad\vw[\obj](\sol)}^2\\
    &~~+ \frac{\nWorkersActive}{\nWorkers}
    \norm{\grad\vw[\obj](\vwt[\state])-\grad\vw[\obj](\sol)}^2.
    \end{aligned}
\end{equation}
Summing from $\worker=1$ to $\nWorkers$, applying \cref{lem:aux}a and taking total expectation, we get
\begin{equation}
    \notag
    \begin{aligned}
    \update[\cmeasure]
    &\le \left(1-\frac{\nWorkersActive}{\nWorkers}\right)\vt[\cmeasure]
    +\frac{\nWorkersActive}{\nWorkers}(2\lips^2\vt[\varX]+4\nWorkers\lips\vt[\errgap])\\
    &= \left(1-\frac{\nWorkersActive}{\nWorkers}\right)\vt[\cmeasure]
    +\frac{2\lips^2\nWorkersActive}{\nWorkers}\vt[\varX]
    + 4\lips\nWorkersActive\vt[\errgap].
    \end{aligned}
\end{equation}
\vskip 3pt
\textbf{Conclude.}
Putting all together we get exactly \eqref{eq:recineq-matrix}.
\end{proof}

From the linear system of inequalities \eqref{eq:recineq-matrix} there are multiple ways to derive the linear convergence of the algorithm.
To obtain the explicit convergence rate and step-size condition presented in \cref{thm:cvg-bistoch}, we construct a suitable Lyapunov function which is a linear combination of $\vt[\distsol], \vt[\varX], \vt[\varY],$ and $\vt[\cmeasure]$ with positive coefficients, and prove that this function decreases geometrically at each iteration.

\begin{proof}[Proof of \cref{thm:cvg-bistoch}]
Let us consider the vector
\begin{equation}
    \notag
    \vvec=\begin{bmatrix}1
        &\frac{\sqrt{\nWorkersActive}(1-\cfactor)}{\nWorkers^{\frac{3}{2}}}
        &\frac{\step(1-\cfactor)}{96\nWorkers\lips}
        &\frac{\step}{12\nWorkers\lips}\end{bmatrix}^{\top},
\end{equation}
and $\cvgrate$ as defined in \cref{thm:cvg-bistoch}, it can be verified that \cref{prop:recineq-matrix} implies 
\begin{equation}
    \label{eq:lyapunov-contraction}
    \vvec^{\top}\update[\recvec]\le\cvgrate~\vvec^{\top}\vt[\recvec]
\end{equation}
whenever step-size condition \eqref{eq:eta-condition} is satisfied.
This means $\vvec^{\top}\vt[\recvec]$ converges geometrically in $\bigoh(\cvgrate^{\run})$.
To conclude, we use the inequality 
\begin{equation}
    \notag
    \ex[\norm{\vwt[\state]-\sol}^2]\le\ex[2\norm{\vwt[\state]-\vt[\avg[\state]]}^2+2\norm{\vt[\avg[\state]]-\sol}^2]
    \le2\vt[\varX]+2\vt[\distsol].
\end{equation}
Detailed computations for proving \eqref{eq:lyapunov-contraction} are provided in \cref{app:proof-lyap}.
\end{proof}

\subsection{Analysis for the general case}
\label{subsec:anal-gen}

Under our weakest set of assumptions (\cref{asm:matrices,asm:sampling}), the mixing matrices do not provide a contraction towards a consensus at each iteration. Nevertheless, the primitivity of the mixing matrices in expectation enables us to show that after a certain number of gossip steps $\contractInt$ (implicitly defined), some sort of contraction happens for both matrices sequences but with respect to a time-varying weighted average instead of a uniform one.  

This has direct consequences on our proof technique since the linear system of equations developed previously has to be modified and in particular extended to track $\contractInt$ successive iterations. With this augmentation, the proof techniques developed before do not hold anymore and we resort to analyzing the spectral radius of the recurrence matrix by perturbation theory arguments when the stepsize is small.

These two points significantly complicate the convergence proof of the method and constitute the main technical contributions of the paper.

\subsubsection{Multi-step contraction}



To establish the multi-step contraction brought by the mixing matrices, we first leverage the primitivity assumption on $\mixing=\ex[\vt[\mixing][\start]]$ and $\mixingalt=\ex[\vt[\mixingalt][\start]]$ to show that inequalities similar to the one in \cref{asm:matrices-bi} hold when we consider the product of successive matrices, which we abbreviate as\footnote{If $\run<\runalt$ we use the notation 
$\vtInt[\mixing][\run][\runalt]=\vtInt[\mixingalt][\run][\runalt]=\Id$.}
\begin{equation}
    \notag
    \vtInt[\mixing][\run][\runalt]=\vt[\mixing][\run]\vt[\mixing][\run-1]\ldots\vt[\mixing][\runalt],~~
    \vtInt[\mixingalt][\run][\runalt]=\vt[\mixingalt][\run]\vt[\mixingalt][\run-1]\ldots\vt[\mixingalt][\runalt].
\end{equation}

The following lemma generalizes \cref{asm:matrices-bi} and is useful for deriving inequalities in the form of \cref{lem:mixing-contraction}.

\begin{lemma}
\label{lem:spectural-multi}
Let \cref{asm:independence,asm:matrices} hold. Then, there exists an integer $\contractInt$ such that
%
\begin{gather*}
    \sradius(\ex[\vtInt[\mixing][\start+\contractInt][\start]^{\top}(\Id-\avgMat)
    \vtInt[\mixing][\start+\contractInt][\start]]) < 1,\\
    \sradius(\ex[(\Id-\avgMat)^{\top}\vtInt[\mixingalt][\start+\contractInt][\start]^{\top}
    \vtInt[\mixingalt][\start+\contractInt][\start](\Id-\avgMat)]) < 1.
\end{gather*}
\end{lemma}
\begin{proof}
We will write $\snorm{\mat}$ and $\Fnorm{\mat}$ respectively for the spectral norm and the Frobenius norm of a matrix $\mat$.
\cref{lem:spectural-multi} is an immediate result of \cite[Prop. 2]{ICH13}, which states that $\ex[\Fnorm{(\Id-\avgMat)\vtInt[\mixing][\start+\contractInt][\start]}^2]$ converges to $0$ at a geometric rate. 
We can thus set $\contractInt$ sufficiently large so that $\ex[\Fnorm{(\Id-\avgMat)\vtInt[\mixing][\start+\contractInt][\start]}^2]<1$,
and the first inequality then follows from that
\begin{equation}
    \notag
    \begin{aligned}
    \sradius(\ex[\vtInt[\mixing][\start+\contractInt][\start]^{\top}(\Id-\avgMat)
    \vtInt[\mixing][\start+\contractInt][\start]])
    &\le \ex[\sradius(\vtInt[\mixing][\start+\contractInt][\start]^{\top}(\Id-\avgMat)
    \vtInt[\mixing][\start+\contractInt][\start])]\\
    &= \ex[\snorm{(\Id-\avgMat)\vtInt[\mixing][\start+\contractInt][\start]}^2]\\
    &\le \ex[\Fnorm{(\Id-\avgMat)\vtInt[\mixing][\start+\contractInt][\start]}^2],
    \end{aligned}
\end{equation}
where we have used the convexity of the spectral radius function $\sradius$ and the fact that the spectral norm of a matrix is bounded from above by its Frobenius norm.

For the second inequality, we observe that the matrices $\seqinf[\vt[\mixingalt^{\top}]]$ have exactly the same assumptions as $\seqinf[\vt[\mixing]]$. Moreover, 
\begin{equation}
    \notag
    \begin{aligned}
    &\sradius(\ex[(\Id-\avgMat)^{\top}\vtInt[\mixingalt][\start+\contractInt][\start]^{\top}
    \vtInt[\mixingalt][\start+\contractInt][\start](\Id-\avgMat)])\\
    &~\le \ex[\sradius((\Id-\avgMat)^{\top}\vtInt[\mixingalt][\start+\contractInt][\start]^{\top}
    \vtInt[\mixingalt][\start+\contractInt][\start](\Id-\avgMat))]\\
    &~= \ex[\norm{\vtInt[\mixingalt][\start+\contractInt][\start](\Id-\avgMat)}^2]\\
    &~= \ex[\norm{(\Id-\avgMat)\vt[\mixingalt^{\top}][\start]\ldots\vt[\mixingalt^{\top}][\start+\contractInt]}^2].
    \end{aligned}
\end{equation}
Hence the same argument applies.
\end{proof}

Another important challenge towards proving a result in the spirit of \cref{lem:mixing-contraction} is that the matrices $\seqinf[\vt[\mixing]]$ (resp. $\seqinf[\vt[\mixingalt]]$) do not have a fixed left (resp. right) Perron vector, and as a consequence there are not predetermined values that the variables should converge to after the mixing matrices are applied.
To overcome this difficulty,
we instead introduce two sequence of random vectors $\seqinf[\vt[\rightEigvec]]$ and $(\vt[\leftEigvec])_{\start\le\run\le\nRuns}$. Here $\nRuns$ is a positive integer fixed in advance. Let $\leftEigvecA$ be the left Perron vector of $\mixing$ such that $\ones^{\top}\leftEigvecA=1$.
These sequences are defined recursively by
\begin{equation}
    \notag
    \begin{aligned}
    \vt[\rightEigvec][\start]=\frac{1}{\nWorkers}\ones, ~
    \update[\rightEigvec] = \vt[\mixingalt]\vt[\rightEigvec];
    ~~
    \vt[\leftEigvec][\nRuns]=\leftEigvecA, ~
    \update[\leftEigvec^{\top}]\vt[\mixing] = \vt[\leftEigvec].
    \end{aligned}
\end{equation}
The sequence $(\vt[\leftEigvec])_{\start\le\run\le\nRuns}$ is defined in a time-reversed manner and mimics the absolute probability sequence \cite{Kolmo36,Touri12} that can be defined for $\seqinf[\vt[\mixing]]$.
However, the above construction gives an explicit expression of $\vt[\leftEigvec]$ which turns out to be useful for our proof.
Also notice that the value of $\vt[\leftEigvec]$ is dependent on the choice of $\nRuns$ though this is implicit from the notation.

Since the $\seqinf[\vt[\mixingalt]]$ are column-stochastic and  the $\seqinf[\vt[\mixing]]$ are row-stochastic, one deduces immediately that both $\seqinf[\vt[\rightEigvec]]$ and $(\vt[\leftEigvec])_{\start\le\run\le\nRuns}$ are sequences of probability vectors.
Moreover, under \cref{asm:independence} we have $\ex[\vt[\leftEigvec]]=\ex[\update[\leftEigvec^{\top}]]\ex[\vt[\mixing]]=\ex[\update[\leftEigvec^{\top}]]\mixing$. By induction we then get
\begin{equation}
    \label{eq:exp-ut}
    \ex[\vt[\leftEigvec]] = \leftEigvecA, ~~~ \forall \run\in\oneto{\nRuns}.
\end{equation}

In the remainder of the section, we will take $\contractInt\ge0$ such that the inequalities of \cref{lem:spectural-multi} are satisfied and define 
\begin{equation}
    \label{eq:defcfactor}
    \begin{aligned}
    \cfactor
    =\max(&\sradius(\ex[\vtInt[\mixing][\start+\contractInt][\start]^{\top}(\Id-\avgMat)
    \vtInt[\mixing][\start+\contractInt][\start]]),\\
    &\sradius(\ex[(\Id-\avgMat)^{\top}\vtInt[\mixingalt][\start+\contractInt][\start]^{\top}
    \vtInt[\mixingalt][\start+\contractInt][\start](\Id-\avgMat)]))
    \end{aligned}
\end{equation}
so that $\cfactor<1$.
The multi-step contraction property is stated as follows.

\begin{lemma}
\label{lem:mixing-contraction-multi}
Let \cref{asm:independence,asm:matrices} hold. Take  $\contractInt$ as in \cref{lem:spectural-multi} and $\cfactor$ from \eqref{eq:defcfactor}. Then,
\begin{flalign*}
    ~~a)~ & \exoft[\norm{(\Id-\avgMat)\vtInt[\mixing][\run+\contractInt][\run]\vt[\jstate]}^2]
    \le \cfactor \norm{\vt[\jstate]-\avgconcat}^2.&\\
    ~~b)~ & \exoft[\norm{\vtInt[\mixingalt][\run+\contractInt][\run]
    (\Id-\vt[\rightEigvec][\run]\ones^\top)\vt[\jgstate][\run]}^2]
    \le \cfactor\norm{
    (\Id-\vt[\rightEigvec][\run]\ones^\top)\vt[\jgstate][\run]}^2.&
\end{flalign*}
\end{lemma}
\begin{proof}
The lemma is proved exactly in the same way as \cref{lem:mixing-contraction}. Just notice that
\begin{equation}
    \notag
    \begin{aligned}
    (\Id-\avgMat)\vtInt[\mixing][\run+\contractInt][\run]\vt[\jstate]
    &= (\Id-\avgMat)\vtInt[\mixing][\run+\contractInt][\run](\Id-\avgMat)\vt[\jstate]\\
    &= (\Id-\avgMat)\vtInt[\mixing][\run+\contractInt][\run](\vt[\jstate]-\avgconcat)
    \end{aligned}
\end{equation}
since $\vtInt[\mixing][\run+\contractInt][\run]$ is row-stochastic. On the other hand,
\begin{equation}
    \notag
    \vtInt[\mixingalt][\run+\contractInt][\run]
    (\Id-\vt[\rightEigvec][\run]\ones^\top)\vt[\jgstate][\run]
    = \vtInt[\mixingalt][\run+\contractInt][\run](\Id-\avgMat)
    (\Id-\vt[\rightEigvec][\run]\ones^\top)\vt[\jgstate][\run].
\end{equation}
since $\vt[\rightEigvec]$ is a probability vector.
\end{proof}

\subsubsection{Linear system of inequalities}

As in \cref{subsec:anal-bistoch}, the proof for \cref{thm:cvg} also relies on the derivation of a linear system of inequalities.
Nevertheless, since there is a contraction only every $\contractInt+1$ steps, we need to take into account the values of relevant quantities for $\contractInt+1$ consecutive iterations and the system becomes $\contractInt+1$ times larger.
Given that the mixing matrices are no longer doubly stochastic, the variables that come into play also need to be modified accordingly.
We consider the following quantities
\begin{equation}
    \label{eq:rec-variable}
    \notag
    \begin{aligned}
    \vt[\alt{\distsol}] &= \ex[\norm{\vt[\leftEigvec^{\top}]\vt[\jstate]-\sol}^2],~~
    \vt[\errgap] = \ex[\obj(\vt[\avg[\state]]) - \obj(\sol)],~~\\
    \vt[\varX] &= \ex[\norm{\vt[\jstate]-\avgconcat}^2],~~
    \vt[\alt{\varY}] = \ex[\norm{\vt[\jgstate]-\vt[\rightEigvec]\ones^{\top}\vt[\jgstate]}],\\
    \vt[\alt{\cmeasure}] &= \ex[\norm{\grad\objMulti(\vt[\jstate])-\grad\objMulti(\vt[\jcstate])}^2].
    \end{aligned}
\end{equation}
Compared to \eqref{eq:rec-variable-bistoch}, 
we define $\vt[\alt{\distsol}]$ because we no longer have $\one^{\top}\vt[\mixing]\vt[\jstate]=\one^{\top}\vt[\jstate]$ while 
it holds $\update[\leftEigvec]\vt[\mixing]\vt[\jstate]=\vt[\leftEigvec]\vt[\jstate]$.
The definition of $\vt[\alt{\varY}]$ is consistent with \cref{lem:mixing-contraction-multi}b.
Finally, we also replace $\vt[\cmeasure]$ by $\vt[\alt{\cmeasure}]$ for technical reasons.
Note that the value of $\vt[\alt{\distsol}]$ depends on $\nRuns$ since its definition involves $\vt[\leftEigvec]$.

The following two lemmas collects several inequalities that will be useful for our proof.
\begin{lemma}
\label{lem:basic}
It holds that
\begin{enumerate}[a)]
    \setlength\topsep{0.3em}
    \setlength\itemsep{0.15em}
    \item $\norm{\vt[\sampMat]}\le1$.
    \item $\norm{\vt[\leftEigvec]^{\top}\vt[\jstate]-\vt[\avg[\state]]}^2
    \le \norm{\vt[\jstate]-\avgconcat}^2$.
    \item The spectral norm of a row- or column-stochastic matrix of size $\nWorkers\times\nWorkers$ is not larger than $\sqrt{\nWorkers}$. 
\end{enumerate}
\end{lemma}
\begin{proof}
\textit{a)} is trivial and \textit{c)} can be proven by using the fact that the spectral norm of a matrix is bounded by its Frobenius norm.
As for \textit{b)}, since $\vt[\leftEigvec]$ is a probability vector,
\begin{equation}
    \notag
    \begin{aligned}
    \norm{\vt[\leftEigvec]^{\top}\vt[\jstate]-\vt[\avg[\state]]}^2
    &= \left\|\sumworker \vwt[\leftEigvecEle]\vwt[\state]-\vt[\avg[\state]]\right\|^2
    \le \sumworker \vwt[\leftEigvecEle]\norm{\vwt[\state]-\vt[\avg[\state]]}^2\\
    &\le \sumworker \norm{\vwt[\state]-\vt[\avg[\state]]}^2
    = \norm{\vt[\jstate]-\avgconcat}^2.
    \end{aligned}
\end{equation}
In the above we have used the notation $\vt[\leftEigvec]=(\vwt[\leftEigvecEle])_{\worker\in\workers}$.
\end{proof}


\begin{lemma}
\label{lem:Gt-bound}
Let \cref{asm:functions} hold and $\seqinf[\vt[\mixingalt]]$ be column-stochastic. We have
\begin{flalign*}
    ~~a)~ & \ex[\norm{\vt[\jvrgvec]-\vt[\rightEigvec]\ones^{\top}\vt[\jvrgvec]}^2]
    \le2\vt[\alt{\varY}] + (4\nWorkers+4)\vt[\alt{\cmeasure}].&\\
    ~~b)~ & \ex[\norm{\vt[\jvrgvec]}^2]\le4\nWorkers\lips^2\vt[\varX]+(8\nWorkers+8)\vt[\alt{\cmeasure}]+4\vt[\alt{\varY}]+8\nWorkers^2\lips\vt[\errgap]. &
\end{flalign*}
\end{lemma}
\begin{proof}
See \cref{app:proof-gt-bound}.
\end{proof}

Since the sampling is not uniform, \cref{lem:yt-gt} does not hold anymore and we need to approximate $\vt[\jvrgvec]$ by $\vt[\rightEigvec]\ones^{\top}\vt[\jvrgvec]$ when deriving the descent inequality. Given the definition of $\vt[\alt{\distsol}]$ and the fact that the nodes are sampled, we say that the \emph{effective step-size} at time $\run$ is $\step\vt[\effstep]$ with
\begin{equation}
    \notag
    \vt[\effstep] = \vt[\leftEigvec]^{\top}\vt[\sampMat]\vt[\rightEigvec].
\end{equation}

The following lemma controls $\ex[\vt[\effstep]\vt[\srvp]]$ for any real-valued non-negative random variable $\vt[\srvp]$ that is $\vt[\filter]$-measurable.

\begin{lemma}
\label{lem:effstep-control}
Let \cref{asm:independence,asm:matrices,asm:sampling} hold. 
We define $\minv{\sampprob}=\min_{\worker\in\workers}\vw[\sampprob]$, $\minv{\leftEigvecAEle}=\min_{\worker\in\workers}[\pi_\mixing]_{\worker}$, and
$\minv{\effstep}=\minv{\leftEigvecAEle}\diaglow\minv{\sampprob}$.
Then, $\minv{\effstep}>0$ and for any $\vt[\filter]$-measurable real-valued non-negative random variable $\vt[\srvp]$, we have
\begin{equation}
    \label{eq:effstep-control}
    \minv{\effstep}\ex[\vt[\srvp]]\le\ex[\vt[\effstep]\vt[\srvp]]\le\ex[\vt[\srvp]].
\end{equation}
\end{lemma}
\begin{proof}
See \cref{app:proof-effstep}.
\end{proof}

We are now ready to state and prove the linear system of inequalities in question. We denote by $P\otimes Q$ the Kronecker product of two matrices $P$ and $Q$, and write
$\bmat_{\worker\workeralt}^{\indg}$ for the matrix of size $\indg\times\indg$ that has a single non-zero entry with value $1$ at position $(\worker,\workeralt)$.

\begin{proposition}
\label{lem:recineq-matrix-general}
For $\nRuns>\contractInt$ and $\run\in\intinterval{1}{\nRuns-\contractInt}$, let $\minv{\effstep}$ be defined as in \cref{lem:effstep-control} and $\vt[\alt{\recvec}]\in\R^{4(\contractInt+1)}$ be defined by
\begin{equation}
    \notag
    \vt[\alt{\recvec}]=
    [\vt[\alt{\distsol}][\run+\contractInt]~ \ldots~ \vt[\alt{\distsol}] ~
    \vt[\varX][\run+\contractInt]~ \ldots~ \vt[\varX] ~
    \vt[\alt{\cmeasure}][\run+\contractInt]~ \ldots~ \vt[\alt{\cmeasure}] ~
    \vt[\alt{\varY}][\run+\contractInt]~ \ldots~ \vt[\alt{\varY}]
    ]^{\top}.
\end{equation}
We also define $\mat_1, \mat_2 \in \R^{(\contractInt+1)\times(\contractInt+1)}$ as
\begin{small}
\begin{gather}
    \notag
    \renewcommand\arraystretch{1.6}
    \mat_1 = \colvec{
    0 & \cdots & \cdots & \cdots & 0 \\
    1 & \ddots  &  &  & \vdots\\
    0 & \ddots & \ddots & & \vdots\\
    \vdots & \ddots & \ddots & \ddots & \vdots\\
    0 & \cdots & 0 & 1 & 0 \\},
    ~~
    \mat_2 = \colvec{
    1 & \cdots & \cdots & \cdots & 1 \\
    0 & \cdots & \cdots  & \cdots & 0 \\
    \vdots & \ddots & \ddots &  & \vdots \\
    \vdots & & \ddots & \ddots & \vdots \\
    0 & \cdots & \cdots & \cdots  & 0 \\
    }.
\end{gather}
\end{small}

Then, under \cref{asm:functions,asm:independence,asm:matrices,asm:sampling}, if \ac{PPDS} is run with $\step\le\minv{\effstep}/(16\nWorkers\lips)$, we have
\begin{equation}
    \label{eq:recineq-matrix-general}
    \update[\alt{\recvec}]\le(\recmat_0+\step\recmat_e)\vt[\alt{\recvec}]
\end{equation}
where 
\begin{small}
\begin{equation}
    \notag
    \begin{aligned}
    \recmat_0 &= \Id_4 \otimes \mat_1 + 
    \cst_{13}\bmat^4_{4,3} \otimes \mat_2
    + \frac{1+\cfactor}{2}(\bmat^4_{1,1}+\bmat^4_{3,3}) \otimes \bmat^{\contractInt+1}_{1,\contractInt+1}\\
    &~+ \left(\bmat^4_{1,1}+\left(1-\frac{\minv{\sampprob}}{2}\right)\bmat^4_{3,3}+\cst_{10}\bmat^4_{3,2}\right) \otimes \bmat^{\contractInt+1}_{1,1}.
    \end{aligned}
\end{equation}
\end{small}
and
\begin{small}
\begin{equation}
    \notag
    \begin{aligned}
    \recmat_e &= 
    \begin{bmatrix}
    -\cst_1 & \cst_2 & \cst_3 & \cst_4 \\
    0 & 0 & 0 & 0 \\
    \cst_{9} & 0 & \cst_{11} & \cst_{12} \\
    0 & 0 & 0 & 0 \\
    \end{bmatrix} \otimes \bmat^{\contractInt+1}_{1,1}
    + \begin{bmatrix}
    0 & 0 & 0 & 0 \\
    \cst_{5} & \cst_{6} & \cst_{7} & \cst_{8} \\
    0 & 0 & 0 & 0 \\
    0 & 0 & 0 & 0 \\
    \end{bmatrix} \otimes \mat_2
    \end{aligned}
\end{equation}
\end{small}
are defined with positive constants $(\cst_{\indg})_{1\le\indg\le13}$ that are entirely determined by $\str,\lips,\nWorkers,\cfactor,\contractInt,\minv{\sampprob},$ and $\minv{\effstep}$.
\end{proposition}
\begin{proof}
We will make use of the inequality
\begin{equation}
    \label{eq:et-bound}
    \vt[\errgap]\le\lips^2(\vt[\alt{\distsol}]+\vt[\varX]).
\end{equation}
This comes from the simple fact that
\begin{equation}
    \notag
    \begin{aligned}
    \obj(\vt[\avg[\state]])-\obj(\sol)
    &\le \frac{\lips^2}{2}\norm{\vt[\avg[\state]]-\sol}^2\\
    &\le \lips^2(\norm{\vt[\leftEigvec^{\top}]\vt[\jstate]-\vt[\avg[\state]]}^2
    +\norm{\vt[\leftEigvec^{\top}]\vt[\jstate]-\sol}^2)\\
    &\le \lips^2(\norm{\vt[\jstate]-\avgconcat}^2
    +\norm{\vt[\leftEigvec^{\top}]\vt[\jstate]-\sol}^2).
    \end{aligned}
\end{equation}

Also notice that $\norm{\vt[\sampMat]\inter[\jgstate]}$ can be bounded as
\begin{equation}
    \label{eq:DtGt-bound}
    \norm{\vt[\sampMat]\inter[\jgstate]}
    =\norm{\vt[\sampMat]\vt[\jvrgvec]}
    \le\norm{\vt[\sampMat]}\norm{\vt[\jvrgvec]}
    \le\norm{\vt[\jvrgvec]}.
\end{equation}

Now, let us fix $\run\in\intinterval{\contractInt+1}{\nRuns-1}$.
We bound $\update[\varX],$ $\update[\alt{\varY}],$ $\update[\alt{\cmeasure}],$ and $\update[\alt{\distsol}]$ in terms of the previous values of these same variables.

\vskip 3pt
\textbf{Bounding $\update[\alt{\distsol}]$.}
We decompose
\begin{equation}
    \notag
    \begin{aligned}
    &\norm{\update[\leftEigvec^{\top}]\update[\jstate]-\sol}^2\\
    &~= \norm{\update[\leftEigvec^{\top}]\vt[\mixing]\vt[\jstate]
        -\step \update[\leftEigvec^{\top}]\vt[\mixing]\vt[\sampMat]\inter[\jgstate]
        -\sol}^2 \\
    &~= \norm{\vt[\leftEigvec^{\top}]\vt[\jstate]-\sol}^2
    + \step^2\norm{\vt[\leftEigvec^{\top}]\vt[\sampMat]\inter[\jgstate]}^2\\
    &~~ - 2\step
    \product{\vt[\leftEigvec^{\top}]\vt[\jstate]-\sol}{\vt[\leftEigvec^{\top}]\vt[\sampMat]\inter[\jgstate]}.
    \end{aligned}
\end{equation}
With \eqref{eq:DtGt-bound}, the second term can be easily bounded using
\begin{equation}
    \notag
    \norm{\vt[\leftEigvec^{\top}]\vt[\sampMat]\inter[\jgstate]}
    \le \norm{\vt[\leftEigvec]}\norm{\vt[\sampMat]\inter[\jgstate]}
    \le \norm{\vt[\jvrgvec]}.
\end{equation}
As for the third term, it can be further decomposed as
\begin{equation}
    \label{eq:descent-product-decompose}
    \begin{aligned}[b]
        &\product{\vt[\leftEigvec^{\top}]\vt[\jstate]-\sol}{\vt[\leftEigvec^{\top}]\vt[\sampMat]\inter[\jgstate]}\\
        &~= \product{\vt[\leftEigvec^{\top}]\vt[\jstate]-\vt[\avg[\state]]}
            {\vt[\leftEigvec^{\top}]\vt[\sampMat]\vt[\jvrgvec]}\\
        &~~+ \product{\vt[\avg[\state]]-\sol}
            {\vt[\leftEigvec^{\top}]\vt[\sampMat](\vt[\jvrgvec]-\vt[\rightEigvec]\ones^{\top}\vt[\jvrgvec])}\\
        &~~+ \product{\vt[\avg[\state]]-\sol}
            {\vt[\leftEigvec^{\top}]\vt[\sampMat]\vt[\rightEigvec]\ones^{\top}\vt[\jvrgvec]},
    \end{aligned}
\end{equation}
where we used again $\vt[\sampMat]\inter[\jgstate]=\vt[\sampMat]\vt[\jvrgvec]$.
Let us bound the three terms separately.
Using \cref{lem:Gt-bound}b, for any $\youngdelta_1>0$, we have
\begin{equation}
    \notag
    \begin{aligned}
    &\ex[-2\step\product{\vt[\leftEigvec^{\top}]\vt[\jstate]-\vt[\avg[\state]]}
            {\vt[\leftEigvec^{\top}]\vt[\sampMat]\vt[\jvrgvec]}]\\
    &~\le \ex[\step \youngdelta_1 \norm{\vt[\leftEigvec^{\top}]\vt[\jstate]-\vt[\avg[\state]]}^2
    + \frac{\step}{\youngdelta_1} \norm{\vt[\leftEigvec^{\top}]\vt[\sampMat]\vt[\jvrgvec]}^2]\\
    &~\le \step \youngdelta_1 \ex[\norm{\vt[\jstate]-\avgconcat}^2] + \frac{\step}{\youngdelta_1} \ex[\norm{\vt[\jvrgvec]}^2]]\\
    &~\le \step \youngdelta_1 \vt[\varX] + \frac{\step}{\youngdelta_1}(4\nWorkers\lips^2\vt[\varX]+(8\nWorkers+8)\vt[\alt{\cmeasure}]+4\vt[\alt{\varY}]+8\nWorkers^2\lips\vt[\errgap]).
    \end{aligned}
\end{equation}
With \cref{lem:basic}b and \cref{lem:Gt-bound}a, we can bound the second term of \eqref{eq:descent-product-decompose} for any $\youngdelta_2>0$ as
\begin{equation}
    \label{eq:Gt-vtGt}
    \begin{aligned}[b]
    &\ex[-2\step\product{\vt[\avg[\state]]-\sol}
            {\vt[\leftEigvec^{\top}]\vt[\sampMat](\vt[\jvrgvec]-\vt[\rightEigvec]\ones^{\top}\vt[\jvrgvec])}]\\
    &~\le \ex[\step \youngdelta_2 \norm{\vt[\avg[\state]]-\sol}^2
    + \frac{\step}{\youngdelta_2} \norm{\vt[\leftEigvec^{\top}]\vt[\sampMat](\vt[\jvrgvec]-\vt[\rightEigvec]\ones^{\top}\vt[\jvrgvec])}^2]\\
    &~\le 2\step \youngdelta_2 \ex[\norm{\vt[\avg[\state]]-\vt[\leftEigvec^{\top}]\vt[\jstate]}^2] 
    + 2\step \youngdelta_2 \ex[\norm{\vt[\leftEigvec^{\top}]\vt[\jstate]-\sol}^2] \\
    &~~+ \frac{\step}{\youngdelta_2} \ex[\norm{\vt[\jvrgvec]-\vt[\rightEigvec]\ones^{\top}\vt[\jvrgvec]}^2]\\
    &\le 2\step \youngdelta_2 \vt[\varX] + 2\step \youngdelta_2 \vt[\alt{\distsol}] 
    + \frac{\step}{\youngdelta_2}(2\vt[\alt{\varY}] + (4\nWorkers+4)\vt[\alt{\cmeasure}]).
    \end{aligned}
\end{equation}
To bound the last term of \eqref{eq:descent-product-decompose}, we use $\ones^{\top}\vt[\jvrgvec]=\sumworker \grad\vw[\obj](\vwt[\state])$. Following \eqref{eq:dt-salar-product}, we then get
\begin{equation}
\notag
\begin{aligned}[b]
    &\product{\vt[\avg[\state]]-\sol}
            {\ones^{\top}\vt[\jvrgvec]}\\
    &~\ge \frac{\nWorkers}{2}(\obj(\vt[\avg[\state]])-\obj(\sol))
    +\frac{\str\nWorkers}{4}\norm{\vt[\avg[\state]]-\sol}^2
    -\frac{\lips}{2}\norm{\vt[\jstate]-\avgconcat}^2\\
    &~\ge \frac{\nWorkers}{2}(\obj(\vt[\avg[\state]])-\obj(\sol))
    +\frac{\str\nWorkers}{8}\norm{\vt[\leftEigvec^{\top}]\vt[\jstate]-\sol}^2\\
    &~~-\frac{\str\nWorkers}{4}\norm{\vt[\avg[\state]]-\vt[\leftEigvec^{\top}]\vt[\jstate]}^2
    -\frac{\lips}{2}\norm{\vt[\jstate]-\avgconcat}^2.
\end{aligned}
\end{equation}
%
Applying \cref{lem:effstep-control} and \cref{lem:basic}b gives
\begin{equation}
    \notag
    \begin{aligned}
    \ex[-2\step\vt[\effstep]
        \product{\vt[\avg[\state]]-\sol}{\ones^{\top}\vt[\jvrgvec]}]
    &\le -\step\minv{\effstep}\nWorkers\vt[\errgap]
    -\frac{\step\minv{\effstep}\str\nWorkers}{4}\vt[\alt{\distsol}]\\
    &~+\step\left(\lips+\frac{\str\nWorkers}{2}\right)\vt[\varX].
    \end{aligned}
\end{equation}
We recall that $\vt[\effstep]=\vt[\rightEigvec]\vt[\sampMat]\vt[\leftEigvec]$.
Putting all together and choosing $\youngdelta_1=16\nWorkers\lips/\minv{\effstep}$ and $\youngdelta_2=\minv{\effstep}\str\nWorkers/16$, we obtain
%
\begin{align*}
    \update[\alt{\distsol}]
    &\le
    \left(1-\frac{\step\minv{\effstep}\str\nWorkers}{4}\right)\vt[\alt{\distsol}]
    -\step\minv{\effstep}\nWorkers\vt[\errgap]\\
    &~+\step\left(\lips+\frac{\str\nWorkers}{2}\right)\vt[\varX]
    +\frac{16\step\nWorkers\lips}{\minv{\effstep}}\vt[\varX]\\
    &~+\left(\frac{\step\minv{\effstep}}{16\nWorkers\lips}+\step^2\right)\\
    &~~~~~\cdot(4\nWorkers\lips^2\vt[\varX]
    +(8\nWorkers+8)\vt[\alt{\cmeasure}]+4\vt[\alt{\varY}]+8\nWorkers^2\lips\vt[\errgap])\\
    &~+\frac{\step\minv{\effstep}\str\nWorkers}{8}\vt[\varX]
    +\frac{\step\minv{\effstep}\str\nWorkers}{8}\vt[\alt{\distsol}]
    +\frac{16\step}{\minv{\effstep}\str\nWorkers}(2\vt[\alt{\varY}] + (4\nWorkers+4)\vt[\alt{\cmeasure}]).
\end{align*}
%
The coefficient of $\vt[\errgap]$ is $-\step\nWorkers\left(\frac{\minv{\effstep}}{2}
    -8\step\nWorkers\lips\right)$.
Since $\step\le\minv{\effstep}/(16\nWorkers\lips)$, this is non-positive and we have indeed
\begin{equation}
    \notag
    \update[\alt{\distsol}]
        \le
        (1-\cst_1\step)\vt[\alt{\distsol}]
        + \cst_2\step \vt[\varX] + \cst_3\step \vt[\alt{\cmeasure}] + \cst_4\step \vt[\alt{\varY}]
\end{equation}
for some positive constants $(\cst_{\indg})_{1\le\indg\le4}$.

\vskip 3pt
\textbf{Bounding $\update[\varX]$.}
Let $\runalt\in\oneto{\run}$. As the matrices $\seqinf[\vt[\mixing]]$ are row-stochastic, it holds
\begin{equation}
    \notag
    \begin{aligned}
    (\Id-\avgMat)\vtInt[\mixing][\run][\runalt+1](\Id-\avgMat)\vt[\mixing][\runalt]
    &= (\Id-\avgMat)\vtInt[\mixing][\run][\runalt+1]\vt[\mixing][\runalt]\\
    &= (\Id-\avgMat)\vtInt[\mixing][\run][\runalt+1]\vt[\mixing][\runalt](\Id-\avgMat).\\
    \end{aligned}
\end{equation}
Hence, for any $\youngdelta>0$, we can write
\begin{small}
\begin{equation}
    \label{eq:row-stoch-step}
    \begin{aligned}[b]
    &\norm{(\Id-\avgMat)\vtInt[\mixing][\run][\runalt+1](\Id-\avgMat)\update[\jstate][\runalt]}^2\\
    &~=\norm{(\Id-\avgMat)\vtInt[\mixing][\run][\runalt+1](\Id-\avgMat)
    \vt[\mixing][\runalt](\vt[\jstate][\runalt]-\step\vt[\sampMat][\runalt]\inter[\jgstate][\runalt])}^2\\
    &~\le
    (1+\youngdelta)
    \norm{(\Id-\avgMat)\vtInt[\mixing][\run][\runalt](\Id-\avgMat)\vt[\jstate][\runalt]}^2\\
    &~~+\left(1+\frac{1}{\youngdelta}\right)\step^2
    \norm{\Id-\avgMat}^2\norm{\vtInt[\mixing][\run][\runalt]}^2\norm{\vt[\sampMat][\runalt]\inter[\jgstate][\runalt]}^2\\
    &~\le
    (1+\youngdelta)
    \norm{(\Id-\avgMat)\vtInt[\mixing][\run][\runalt](\Id-\avgMat)\vt[\jstate][\runalt]}^2
    +\left(1+\frac{1}{\youngdelta}\right)\step^2\nWorkers\norm{\vt[\jvrgvec][\runalt]}^2.
    \end{aligned}
\end{equation}
\end{small}
In the last line we have used the fact that $\vtInt[\mixing][\run][\runalt]$ is row-stochastic so that  $\norm{\vtInt[\mixing][\run][\runalt]}\le\sqrt{\nWorkers}$ and the inequality $\norm{\vt[\sampMat][\runalt]\inter[\jgstate][\runalt]}\le\norm{\vt[\jvrgvec][\runalt]}$.
%

Since $\update[\jstate]-\avgconcat[\state][\run+1]=(\Id-\avgMat)\vtInt[\mixing][\run][\run+1](\Id-\avgMat)\update[\jstate]$, applying \eqref{eq:row-stoch-step} repeatedly then gives
\begin{equation}
    \notag
    \begin{aligned}
    \norm{\update[\jstate]-\avgconcat[\state][\run+1]}^2
    &\le (1+\youngdelta)^{\contractInt+1}\norm{(\Id-\avgMat)\vtInt[\mixing][\run][\run-\contractInt](\Id-\avgMat)\vt[\jstate][\runalt]}^2\\
    &~+ \step^2\nWorkers\left(1+\frac{1}{\youngdelta}\right)
    \sum_{\runalt=0}^{\contractInt} (1+\youngdelta)^\runalt \norm{\vt[\jvrgvec][\run-\runalt]}^2.
    \end{aligned}
\end{equation}
Let $\youngdelta=\frac{1}{\contractInt+1}\log\frac{1+\cfactor}{2\cfactor}>0$ so that for all $0\le\runalt\le\contractInt+1$, we have $(1+\youngdelta)^\runalt\le\frac{1+\cfactor}{2\cfactor}<\frac{1}{\cfactor}$.
Taking expectation in the above inequality and invoking \cref{lem:mixing-contraction-multi}a and \cref{lem:Gt-bound}b leads to
\begin{equation}
    \notag
    \begin{aligned}
    \update[\varX]
    \le \frac{1+\cfactor}{2}\vt[\varX][\run-\contractInt]
    +\frac{\step^2\nWorkers}{\cfactor}\left(1+\frac{1}{\youngdelta}\right)
    \sum_{\runalt=0}^{\contractInt}\vt[\Delta][\run-\runalt],
    \end{aligned}
\end{equation}
where $\vt[\Delta][\run-\runalt]=
    8\nWorkers^2\lips\vt[\errgap][\run-\runalt]+
    4\nWorkers\lips^2\vt[\varX][\run-\runalt]
    +(8\nWorkers+8)\vt[\alt{\cmeasure}][\run-\runalt]
    +4\vt[\alt{\varY}][\run-\runalt]$.
With \eqref{eq:et-bound} and $\step\le\minv{\effstep}/(16\nWorkers\lips)$ we thus see there exist positive constants $(\cst_{\indg})_{5\le\indg\le8}$ such that
\begin{equation}
    \notag
    \update[\varX]
    \le
    \frac{1+\cfactor_1}{2}\vt[\varX]
    + \step\sum_{\runalt=0}^{\contractInt}
    (\cst_5  \vt[\alt{\distsol}][\run-\runalt] + \cst_6\vt[\varX][\run-\runalt]
    + \cst_7 \vt[\alt{\cmeasure}][\run-\runalt] + \cst_8 \vt[\alt{\varY}][\run-\runalt]).
\end{equation}

\vskip 3pt
\textbf{Bounding $\update[\alt{\cmeasure}]$.}
By Young's inequality,
\begin{small}
\begin{equation}
    \notag
    \begin{aligned}
    \norm{\grad\vw[\obj](\vwtupdate[\cstate])-\grad\vw[\obj](\vwtupdate[\state])}^2
    &\le \frac{2-\vw[\sampprob]}{2-2\vw[\sampprob]}
    \norm{\grad\vw[\obj](\vwtupdate[\cstate])-\grad\vw[\obj](\vwt[\state])}^2\\
    &+ \frac{2-\vw[\sampprob]}{\vw[\sampprob]}
    \norm{\grad\vw[\obj](\vwt[\state])-\grad\vw[\obj](\vwtupdate[\state])}^2.
    \end{aligned}
\end{equation}
\end{small}
The update rule of $\vwt[\cstate]$ implies that
\begin{equation}
    \notag
    \begin{aligned}
    &\exoft[\norm{\grad\vw[\obj](\vwtupdate[\cstate])-\grad\vw[\obj](\vwt[\state])}^2]\\
    &~= (1-\vw[\sampprob])\norm{\grad\vw[\obj](\vwt[\cstate])-\grad\vw[\obj](\vwt[\state])}^2\\
    &~~+ \vw[\sampprob]
    \norm{\grad\vw[\obj](\vwt[\state])-\grad\vw[\obj](\vwt[\state])}^2\\
    &~= (1-\vw[\sampprob])\norm{\grad\vw[\obj](\vwt[\cstate])-\grad\vw[\obj](\vwt[\state])}^2.
    \end{aligned}
\end{equation}
With $\minv{\sampprob}=\min_{\worker\in\workers}\vw[\sampprob]$ as defined in \cref{lem:effstep-control}, we then have
\begin{equation}
    \notag
    \begin{aligned}
    &\exoft[\norm{\grad\vw[\obj](\vwtupdate[\cstate])-\grad\vw[\obj](\vwtupdate[\state])}^2]\\
    &~\le \left(1-\frac{\vw[\sampprob]}{2}\right)
    \norm{\grad\vw[\obj](\vwt[\cstate])-\grad\vw[\obj](\vwt[\state])}^2\\
    &~~+ \frac{2}{\vw[\sampprob]}
    \exoft[\norm{\grad\vw[\obj](\vwt[\state])-\grad\vw[\obj](\vwtupdate[\state])}^2]\\
    &~\le \left(1-\frac{\minv{\sampprob}}{2}\right)
    \norm{\grad\vw[\obj](\vwt[\cstate])-\grad\vw[\obj](\vwt[\state])}^2\\
    &~~+ \frac{2}{\minv{\sampprob}}
    \exoft[\norm{\grad\vw[\obj](\vwt[\state])-\grad\vw[\obj](\vwtupdate[\state])}^2].
    \end{aligned}
\end{equation}
Taking total expectation and summing from $\worker=1$ to $\nWorkers$ gives
\begin{equation}
    \notag
    \update[\alt{\cmeasure}]
    \le \left(1-\frac{\minv{\sampprob}}{2}\right)\vt[\alt{\cmeasure}]
    +\frac{2}{\minv{\sampprob}}\ex[\norm{\grad\objMulti(\vt[\jstate])-\grad\objMulti(\update[\jstate])}^2].
\end{equation}
By Lipschitz-continuity of the gradients, it holds $\norm{\grad\objMulti(\vt[\jstate])-\grad\objMulti(\update[\jstate])}\le\lips\norm{\vt[\jstate]-\update[\jstate]}$.
We then develop
\begin{equation}
    \notag
    \begin{aligned}
    \update[\jstate]-\vt[\jstate]
    &= \vt[\mixing](\vt[\jstate]-\step\vt[\sampMat]\inter[\jgstate]) - \vt[\jstate]\\
    &= (\vt[\mixing]-\Id)(\Id-\avgMat)\vt[\jstate] - \step \vt[\mixing]\vt[\sampMat]\inter[\jgstate].
    \end{aligned}
\end{equation}
With $\norm{\vt[\mixing]-\Id}^2\le2\norm{\vt[\mixing]}^2+2\norm{\Id}^2\le2\nWorkers+2$, we obtain that
\begin{equation}
    \notag
    \begin{aligned}
    &\norm{\grad\objMulti(\vt[\jstate])-\grad\objMulti(\update[\jstate])}^2\\
    &~\le \lips^2((4\nWorkers+4)\norm{\vt[\jstate]-\avgconcat[\state]}^2+2\step^2\nWorkers\norm{\vt[\jvrgvec]}^2).
    \end{aligned}
\end{equation}
Combining the above and applying \cref{lem:Gt-bound}b leads to
\begin{equation}
    \notag
    \begin{aligned}
    \update[\alt{\cmeasure}]
    &\le \left(1-\frac{\minv{\sampprob}}{2}\right)\vt[\alt{\cmeasure}]
    +\frac{2\lips^2}{\minv{\sampprob}}
    ((4\nWorkers+4)\vt[\varX]\\
    &~+2\step^2\nWorkers
    (4\nWorkers\lips^2\vt[\varX]+(8\nWorkers+8)\vt[\alt{\cmeasure}]+4\vt[\alt{\varY}]+8\nWorkers^2\lips\vt[\errgap])).
    \end{aligned}
\end{equation}
Using \eqref{eq:et-bound} and $\step\le\minv{\effstep}/(16\nWorkers\lips)$ we deduce the existence of positive constants $(\cst_{\indg})_{9\le\indg\le12}$ such that
\begin{equation}
    \notag
    \update[\alt{\cmeasure}] \le
    \cst_9\step  \vt[\alt{\distsol}] + \cst_{10} \vt[\varX]
        + \left(1-\frac{\minv{\sampprob}}{2}+\cst_{11}\step\right) + \cst_{12}\step \vt[\alt{\varY}].
\end{equation}

\vskip 3pt
\textbf{Bounding $\update[\alt{\varY}]$.}
Let $\runalt\in\oneto{\run}$. Using the column-stochasticity of $\vt[\mixingalt][\runalt]$ and the definition $\update[\rightEigvec][\runalt]=\vt[\mixingalt][\runalt]\vt[\rightEigvec][\runalt]$, we get
\begin{equation}
    \notag
    (\Id-\update[\rightEigvec][\runalt]\ones^\top)\vt[\mixingalt][\runalt]
    = \vt[\mixingalt][\runalt] - \update[\rightEigvec][\runalt]\ones^\top
    = \vt[\mixingalt][\runalt] - \vt[\mixingalt][\runalt]\vt[\rightEigvec][\runalt]\ones^\top.
\end{equation}
Hence, for any $\youngdelta>0$, it holds that
\begin{small}
\begin{align}
    \notag
    &\norm{\vtInt[\mixingalt][\run][\runalt+1](\Id-\update[\rightEigvec][\runalt]\ones^\top)\update[\jgstate][\runalt]}^2
    \\
    \notag
    &= \norm{\vtInt[\mixingalt][\run][\runalt+1](\Id-\update[\rightEigvec][\runalt]\ones^\top)
    \vt[\mixingalt][\runalt](\vt[\jgstate][\runalt]+\vt[\sampMat][\runalt](\grad\objMulti(\vt[\jstate][\runalt])-\grad\objMulti(\vt[\jcstate][\runalt])))}^2
    \\
    \notag
    &\le
    (1+\youngdelta)
    \norm{\vtInt[\mixingalt][\run][\runalt](\Id-\vt[\rightEigvec][\runalt]\ones^\top)\vt[\jgstate][\runalt]}^2
    \\
    \notag
    &~+\left(1+\frac{1}{\youngdelta}\right)
    \norm{\vtInt[\mixingalt][\run][\runalt]-\vtInt[\mixingalt][\run][\runalt]\vt[\rightEigvec][\runalt]\ones^\top}^2
    \norm{\vt[\sampMat][\runalt]}^2
    \norm{\grad\objMulti(\vt[\jstate][\runalt])-\grad\objMulti(\vt[\jcstate][\runalt])}^2
    \\
    \notag
    &\le
    (1+\youngdelta)
    \norm{\vtInt[\mixingalt][\run][\runalt](\Id-\vt[\rightEigvec][\runalt]\ones^\top)\vt[\jgstate][\runalt]}^2\\
    &~+ 4\nWorkers\left(1+\frac{1}{\youngdelta}\right)
    \norm{\grad\objMulti(\vt[\jstate][\runalt])-\grad\objMulti(\vt[\jcstate][\runalt])}^2.
    \label{eq:column-stoch-step}
\end{align}
\end{small}
In the last inequality we have used
\begin{equation}
    \notag
    \norm{\vtInt[\mixingalt][\run][\runalt]-\vtInt[\mixingalt][\run][\runalt]\vt[\rightEigvec][\runalt]\ones^\top}
    \le\norm{\vtInt[\mixingalt][\run][\runalt]}+\norm{\vtInt[\mixingalt][\run][\runalt]\vt[\rightEigvec][\runalt]\ones^\top}\le2\sqrt{\nWorkers}
\end{equation}
which is true because both $\vtInt[\mixingalt][\run][\runalt]$ and $\vtInt[\mixingalt][\run][\runalt]\vt[\rightEigvec][\runalt]\ones^\top$ are column-stochastic.

Since $\update[\jgstate]-\update[\rightEigvec]\ones^\top\update[\jgstate]=\vtInt[\mixingalt][\run][\run+1](\Id-\update[\rightEigvec]\ones^\top)\update[\jgstate]$, applying \eqref{eq:column-stoch-step} repeatedly then gives
\begin{equation}
    \notag
    \begin{aligned}[b]
    &\norm{\update[\jgstate]-\update[\rightEigvec]\ones^\top\update[\jgstate]}^2\\
    &~\le (1+\youngdelta)^{\contractInt+1}
    \norm{\vtInt[\mixingalt][\run][\run-\contractInt]
    (\Id-\vt[\rightEigvec][\run-\contractInt]\ones^\top)\vt[\jgstate][\run-\contractInt]}^2\\
    &~~+4\nWorkers\left(1+\frac{1}{\youngdelta}\right)
    \sum_{\runalt=0}^{\contractInt} (1+\youngdelta)^\runalt
    \norm{\grad\objMulti(\vt[\jstate][\run-\runalt])-\grad\objMulti(\vt[\jcstate][\run-\runalt])}^2.
    \end{aligned}
\end{equation}
%
Let us take $\youngdelta=\frac{1}{\contractInt+1}\log\frac{1+\cfactor}{2\cfactor}>0$ as before. Taking total expectation in the above inequality and invoking \cref{lem:mixing-contraction-multi}b leads to
\begin{equation}
    \notag
    \begin{aligned}
    \update[\alt{\varY}]
    &\le
    \frac{1+\cfactor_2}{2}\vt[\alt{\varY}][\run-\contractInt]
    + \frac{4\nWorkers}{\cfactor}\left(1+\frac{1}{\youngdelta}\right)
    \sum_{\runalt=0}^{\contractInt} \vt[\alt{\cmeasure}][\run-\runalt].
    \end{aligned}
\end{equation}
We set $\cst_{13} = \frac{4\nWorkers}{\cfactor}\left(1+\frac{1}{\youngdelta}\right)$.

\vskip 3pt
\textbf{Conclude.}
Putting all together we get exactly \eqref{eq:recineq-matrix-general}.
\end{proof}

\subsubsection{Geometric convergence of PPDS}
From \cref{lem:recineq-matrix-general}, we are now in position to prove the geometric convergence of \ac{PPDS} by showing that the spectral radius of $\recmat_0+\step\recmat_e$ is smaller than $1$ for $\step>0$ sufficiently small.

\begin{proof}[Proof of \cref{thm:cvg}]
In the following we analyze the eigenvalues of $\recmat_0+\step\recmat_e$ with help of matrix perturbation theory.
We first notice that $\recmat_0$ is a block-triangular matrix. Its characteristic polynomial can be easily computed and is given\;by 
\begin{equation}
\notag
P_{\recmat_0}(\scalar)=\scalar^{2\contractInt}(\scalar-1)
\left(\scalar-\left(1-\frac{\minv{\sampprob}}{2}\right)\right)
\left(\scalar^{\contractInt+1}-\frac{1+\cfactor}{2}\right)^2.
\end{equation}
This shows that the spectral radius of $\recmat_0$ is $1$ and $1$ is also the unique eigenvalue of largest modulus of the matrix.

Let us denote by $\eigv_1=1, \eigv_2, \ldots, \eigv_{4(\contractInt+1)}$ the eigenvalues of $\recmat_0$ so that $\abs{\eigv_\indg}<1$ for all $\indg\in\intinterval{2}{4(\contractInt+1)}$.
By continuity of the eigenvalues, for any $\smallradius>0$ there exists $\delta>0$ such that if $\step<\delta$, for any $\eigv_\indg$ of multiplicity $m$ the matrix $\recmat_0+\step\recmat_e$ has exactly $m$ eigenvalues (counting multiplicity) in $\ball(\eigv_\indg,\smallradius)$, the open disk centered at $\eigv_\indg$ with radius $\smallradius$; see  \cite[Chap. 5.1]{Kato13}.
Let us take $\smallradius$ small enough such that all the eigenvalues of $\recmat_0+\step\recmat_e$ are smaller than $1-\smallradius$ in modulus except the greatest one. 
For $\step<\delta$, we can then define $\eigv_1(\step)$ as the unique eigenvalue of $\recmat_0+\step\recmat_e$ that is in $\ball(1,\smallradius)$. We will now show that $\abs{\eigv_1(\step)}<1$ for $\step$ sufficiently small.
For this, let
\begin{equation}
    \notag
    \leftEigvec = [1 ~ 0 ~ \ldots ~ 0]^{\top},
    ~~ \rightEigvec = [\underbrace{1 ~ \ldots ~ 1}_{\contractInt+1~ \text{times}}~ 0 ~\ldots~ 0]^{\top}
\end{equation}
be respectively the left and the right eigenvector of $\recmat_0$ associated with the eigenvalue $1$.
By \cite[Th. 6.3.12]{JH13} (see also \cite[Th. 1]{GLO20}), we have
\begin{equation}
    \notag
    \eigv_1'(0) = \frac{\leftEigvec^{\top}\recmat_e\rightEigvec}{\leftEigvec^{\top}\rightEigvec}=-\cst_1<0.
\end{equation}
As a consequence, $\abs{\eigv_1(\step)}<1$ for $\step$ sufficiently small and subsequently $\sradius(\recmat_0+\step\recmat_e)<1$. 

In order to conclude, we need to get rid of the dependence on $\nRuns$ which plays a role in the definition of the vectors $(\vt[\leftEigvec])_{1\le\run\le\nRuns}$ and the quantities $(\vt[\alt{\distsol}])_{1\le\run\le\nRuns}$.
We recall that $\vt[\distsol]=\ex[\norm{\vt[\avg[\state]]-\sol}^2]$.
As in \eqref{eq:Gt-vtGt}, we have both $\vt[\distsol]\le2\vt[\alt{\distsol}]+2\vt[\varX]$ and $\vt[\alt{\distsol}]\le2\vt[\distsol]+2\vt[\varX]$.
Let us define $\vt[\recvec'']$ by replacing $(\vt[\alt{\distsol}][\runalt])_{\run\le\runalt\le\run+\contractInt}$ by $(\vt[\distsol][\runalt])_{\run\le\runalt\le\run+\contractInt}$ in $\vt[\alt{\recvec}]$.
The above inequalities can then be translated into $\vt[\recvec'']\le \mat\vt[\alt{\recvec}]$ and $\vt[\alt{\recvec}]\le \mat\vt[\recvec'']$ for a non-negative matrix $\mat$ properly defined.
Note that neither $\mat$ nor $\recmat_0+\step\recmat_e$ depend on $\run$ or $\nRuns$.
Therefore, the inequality
\begin{equation}
    \notag
    \vt[\recvec'']
    \le \mat (\recmat_0+\step\recmat_e)^{\run-1} \mat \vt[\recvec''][1]
\end{equation}
which holds for all $\run\in\N$ guarantees the geometric convergence of $ \vt[\recvec''] $ and subsequently of all the relevant quantities when $\step$ is small enough.

Finally, from the geometric convergence of $\ex[\norm{\vwt[\state]-\sol}^2]$, we deduce that $\vwt[\state]$ converges to $\sol$ almost surely by using Markov's inequality and the Borel–Cantelli lemma.
\end{proof}

\section{Simulations}
\label{sec:sim}
In this section, we illustrate the interest of PPDS for asynchronous decentralized optimization on a) a synthetic ridge regression problem; and b) a logistic regression problem on a real dataset. 
Ablation study of the how different network parameters influence the performance of PPDS is provided in \cref{apx:exp-setup}.\footnote{%
The code to reproduce the experiments can be found at~\url{https://github.com/yassine-laguel/ppds}.}

\begin{figure*}[htp]
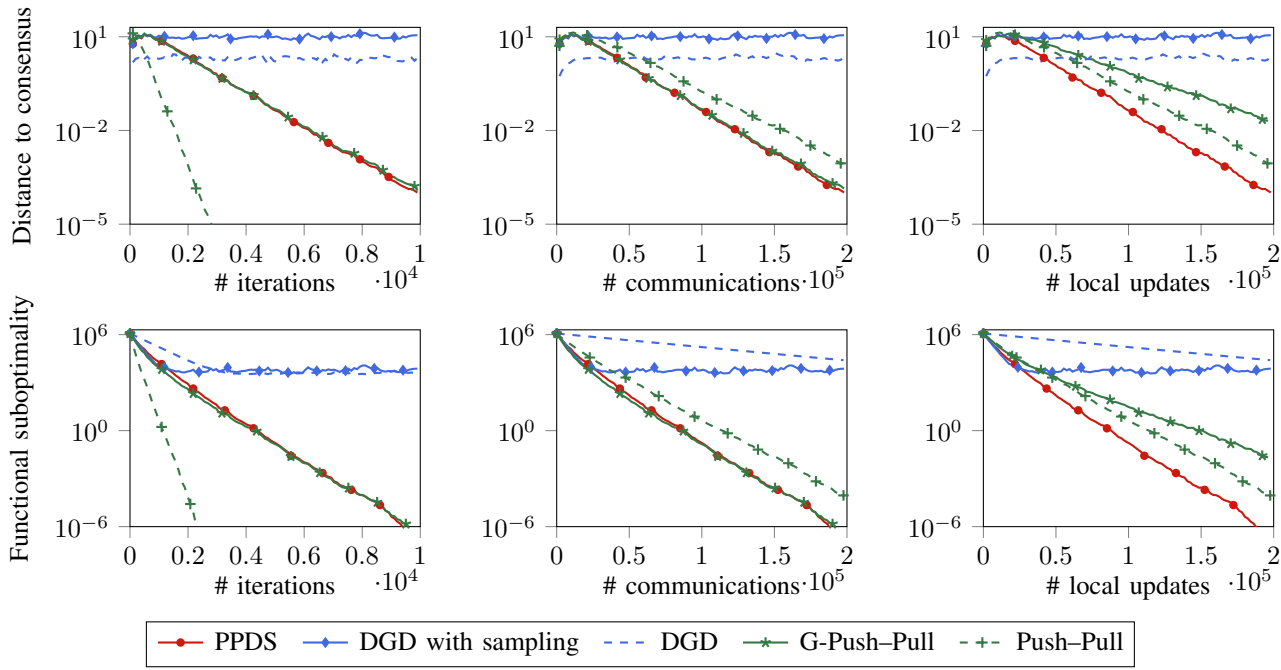

  \centering
\definecolor{color0}{rgb}{0.8,0.101960784313725,0.0549019607843137}
\definecolor{color1}{rgb}{0.254901960784314,0.411764705882353,0.882352941176471}
\definecolor{color2}{rgb}{0.215686274509804,0.480392156862745,0.272549019607843}

\begin{tikzpicture}

\begin{groupplot}[group style={group size=3 by 2,horizontal sep=0.1\textwidth,vertical sep=4em},
ylabel style={yshift=0.6em},
height=4.2cm,
width = 0.3\textwidth,
tick align=outside,
tick pos=left,
ymin=1e-5, ymax=20,
ymode=log,
]

\input{Content/figs/synthetic_consensus.tikz}
\input{Content/figs/synthetic_losses.tikz}

\end{groupplot}
\end{tikzpicture}

\vspace{0.5em}

\begin{tikzpicture}

\matrix[
    matrix of nodes,
    anchor=north,
    draw,
    inner sep=0.2em,
    draw
  ]
  {
    \ref{plots:plot1}& PPDS &[5pt]
    \ref{plots:plot2}& DGD with sampling &[5pt]
    \ref{plots:dashed1}& DGD  &[5pt]
    \ref{plots:plot3}& G-Push--Pull &[5pt]
    \ref{plots:dashed2}& Push--Pull  \\};
\end{tikzpicture}

     \caption{\small{Numerical illustrations for Ridge regression on synthetic dataset.}}
     \label{fig:numerical_illustration}
\end{figure*}

\begin{table}[h!]
\caption{\small{Datasets and graph description}\label{table:datasets}}
\centering
\begin{tabular}{lcc}
\toprule
       & Synthetic & EMNIST \\
\midrule
Number of features $\vdim$   & $10$      & $784$ \\
Number of examples      & $10000$      & $2500$ \\
\midrule
Devices $\nWorkers$   & $100$      & $50$ \\
Local size $n_{\text{local}}$   & $100$      & $50$ \\
RGG radius $r$   & $0.2$      & $0.3$ \\
Sampled nodes / round   & $20$      & $10$ \\
Sampled neighbors / communication   & $1$      & $1$ \\
\bottomrule
\end{tabular}
\end{table}

\subsection{Dataset, tasks and models}
\label{sec:num:datasets}

For both problems, we minimize an objective of the form
$$
\obj(\point) =
    \frac{1}{\nWorkers} \sumworker  \underbrace{ \left( \sum_{j=1}^{n_{\text{local}}} \obj_{\worker,j}(\point) + \vw[\lambda]\|\point\|_2^2 \right) }_{\vw[\obj](\point)}
$$
that is, each worker $\worker$ has a local dataset of $n_{\text{local}}$ examples and a Tikhonov regularization term with parameter  $\vw[\lambda] = 1/n_{\text{local}}$. Since the local objectives are convex, this regularization make the problem strongly convex. 

We form the communication networks by generating Random Geometric Graphs (RGG) using the library \texttt{networkx} \cite{hagberg2008exploring} with different number of nodes $\nWorkers$ and radius $r$ for each experiment. 
We consider the \textit{broadcast} setting illustrated in the third example of \cref{sec:special}.
Precisely, all activated nodes broadcast their models to 
one (randomly chosen) 
neighbor during a communication step.
We illustrate the effect of device sampling by comparing algorithms with \emph{full-device participation} and with \emph{random device sampling} ($20$ nodes for the synthetic dataset and $10$ nodes for EMNIST). The relevant parameters are reported in Table~\ref{table:datasets}.


\paragraph{Ridge regression on synthetic dataset} For this problem, the local losses are defined as
$$
    \obj_{\worker,j}(\point) = (\vw[b][\worker,j] - \point^\top \vw[a][\worker,j])^2
$$
where $(\vw[a][\worker,j], \vw[b][\worker,j]) \in  \vecspace \times \mathbb{R}$ are data points generated using the procedure \texttt{make\_regression} from \texttt{scikit-learn} \cite{sklearn_api}. Different seeds are used for different nodes, yielding statistically heterogeneous distributions between the nodes. 

\paragraph{Logistic regression on EMNIST} 
The EMNIST dataset \cite{cohen2017emnist} is comprised of images of handwritten digits and letters from several authors. We consider the problem of finding which character is written from its image. For this, we consider local losses of the form 
$$
  \obj_{\worker,j}(\point) = -  \vw[b][\worker,j] \log\left(  \softmax(\point^\top \vw[a][\worker,j]) \right)
$$		
where the $(\vw[a][\worker,j], \vw[b][\worker,j])$ are respectively $\vdim = 28\times 28$ gray-scale images of handwritten digits character and their associated one-hot label $\vw[b][\worker,j]\in (0-9, a-z, A-Z)$ totaling 62 classes. 
Each worker's local dataset comes from images from the same author.

\subsection{Algorithms, hyperparameters and evaluation metrics}\label{sec:num:algorithms}

We compare the proposed algorithm PPDS with several baselines: Decentralized Gradient Descent (DGD) with and without sampling, Push--Pull and G-Push--Pull.
The same broadcast communication scheme is applied to all the methods, and the same uniform sampling strategy is adopted whenever device sampling is involved.


\begin{figure*}[htp]
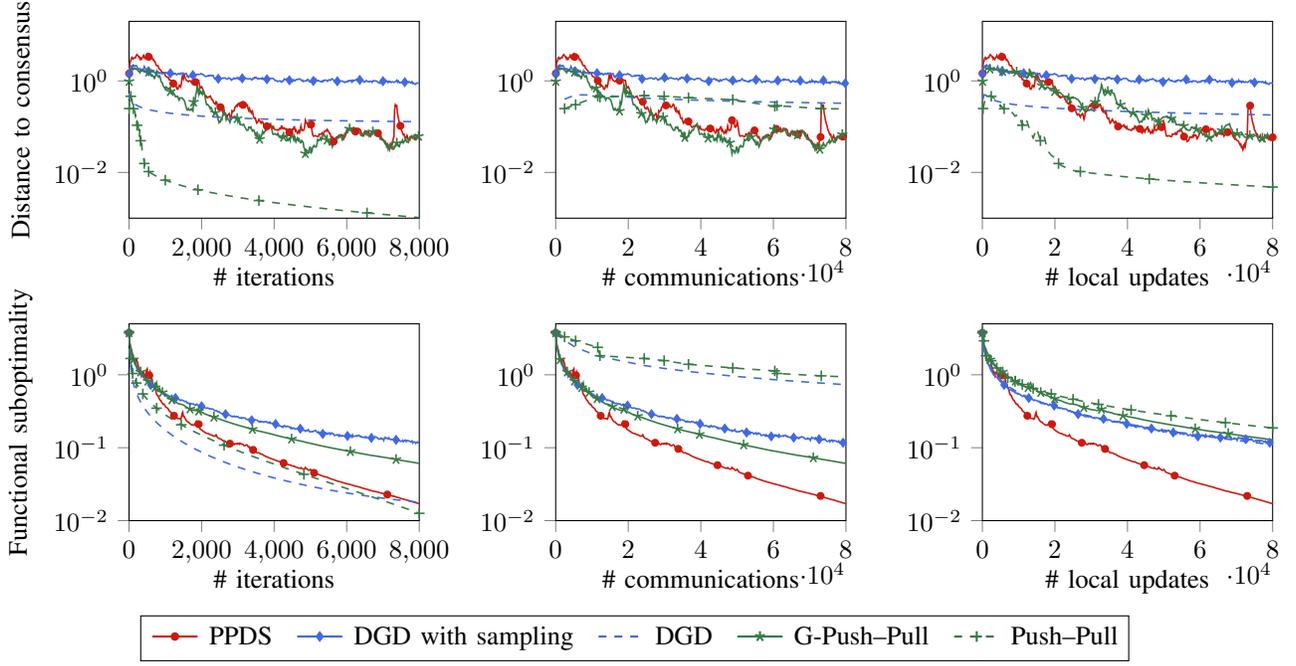

  \centering
\definecolor{color0}{rgb}{0.8,0.101960784313725,0.0549019607843137}
\definecolor{color1}{rgb}{0.254901960784314,0.411764705882353,0.882352941176471}
\definecolor{color2}{rgb}{0.215686274509804,0.480392156862745,0.272549019607843}

\begin{tikzpicture}

\begin{groupplot}[group style={group size=3 by 2,horizontal sep=0.1\textwidth,vertical sep=4em},
width = 0.3\textwidth,
height=4.2cm,
ylabel style={yshift=0.6em},
tick align=outside,
tick pos=left,
ymin=1e-3, ymax=20,
ymode=log,
]

\input{Content/figs/emnist_consensus.tikz}

\input{Content/figs/emnist_losses.tikz}

\end{groupplot}

\end{tikzpicture}

\vspace{0.5em}

\begin{tikzpicture}

\matrix[
    matrix of nodes,
    anchor=north,
    draw,
    inner sep=0.2em,
    draw
  ]
  {
    \ref{plots:plot1}& PPDS &[5pt]
    \ref{plots:plot2}& DGD with sampling &[5pt]
    \ref{plots:dashed1}& DGD  &[5pt]
    \ref{plots:plot3}& G-Push--Pull &[5pt]
    \ref{plots:dashed2}& Push--Pull  \\};
\end{tikzpicture}
     \caption{\small{Numerical illustrations for logistic regression on EMNIST.}}
     \label{fig:numerical_illustration2}
\end{figure*}

For each method, the stepsize is taken fixed, tuned with a coarse-to-fine strategy: we first select the stepsize $\step$ within $\{10^{-k}, 2\leq k\leq 5 \}$ yielding the best global training loss;
then, a second search is performed over $\{\eta_1 2 ^{k}, -2\leq k \leq 2 \}$. 

We report two evaluation metrics: 
\begin{enumerate*}[\textup(\itshape i\textup)]
    \item the distance to consensus $(1/\nWorkers)\sum_{\worker=1}^{\nWorkers}\norm{\vwt[\state]-\vt[\avg[\state]]}^2$; and
    \item  the functional suboptimality $(1/\nWorkers)\sum_{\worker=1}^\nWorkers \obj(\vwt[\state]) - \sol[\obj]$.
\end{enumerate*}
For the synthetic dataset, the optimal solution is computed by inversion of a linear system. For the real dataset, it is set as the best solution in terms of final training losses found by the implemented algorithms. 

For each experiment, we report these metrics in terms of three different measures: \begin{enumerate*}[\textup(\itshape i\textup)]
\item number of iterations;
\item communication cost, \ie the cumulative number of activated communication links; and
\item number of local updates.
\end{enumerate*}
These cover different aspects that influence the efficiency of distributed optimization.

\subsection{Numerical results}\label{sec:num:results}


First, we observe on \cref{fig:numerical_illustration,fig:numerical_illustration2} that the proposed method PPDS converges linearly,
as expected from \cref{thm:cvg}. Furthermore, looking at the right-hand plots, we see that PPDS outperforms all the other methods when it comes to measuring the functional optimality with respect to the number of local updates.
This illustrates that PPDS indeed saves computational resources, by an efficient interplay between computation and communications.
Concerning the communication complexity (middle plots), PPDS is at least as competitive as 
G-Push--Pull, which was shown in \cite{PSXN20} to beat other baselines. 
Finally, we observe that Push--Pull, as a synchronous gradient-tracking method, naturally achieves the best performance when measuring in terms of the number of iterations (left-hand plots), but tends to be less efficient when we consider the actual communication and computational costs. 



\section{Conclusions}

In this paper, we showed how device sampling can be incorporated in asynchronous decentralized gradient descent, by extending the Push--Pull method. We proved linear convergence of the method on strongly convex functions and validate our approach on problems with synthetic and real data.
This work also opens towards several research directions. 
This goes from the theoretical analysis of our method in non-strongly-convex, non-convex, or/and the stochastic setup (see \eg \cite{koloskova2021improved} for analysis of gradient tracking in these setups)
to the investigation of how our approach can be combined with other existing techniques such as local gradients computation and gradient compression.

\appendix
\subsection{Proof of \cref{lem:aux}}
\label{app:proof-aux}

\begin{proof}
\textit{a)}
Since $\vw[\obj]$ is $\lips$-smooth, it holds for all $\point, \alt{\point}\in\vecspace$ that
\begin{small}
\begin{equation}
    \notag
    \norm{\grad\vw[\obj](\point)-\grad\vw[\obj](\alt{\point})}^2
    \le2\lips(\vw[\obj](\point)-\vw[\obj](\alt{\point})
    -\product{\point-\alt{\point}}{\grad\vw[\obj](\alt{\point})}).
\end{equation}
\end{small}
Subsequently,
\begin{small}
\begin{align*}
    &\sumworker\norm{\grad\vw[\obj](\vwt[\state])-\grad\vw[\obj](\sol)}^2\\
    &~\le 2\sumworker\norm{\grad\vw[\obj](\vwt[\state])-\grad\vw[\obj](\vt[\avg[\state]])}^2
    + 2 \sumworker\norm{\grad\vw[\obj](\vt[\avg[\state]])-\grad\vw[\obj](\sol)}^2\\
    &~\le 2\lips^2\sumworker\norm{\vwt[\state]-\vt[\avg[\state]]}^2\\
    &~~+ 4 \lips\sumworker (\vw[\obj](\vt[\avg[\state]])-\vw[\obj](\sol)
    -\product{\vt[\avg[\state]]-\sol}{\grad\vw[\obj](\sol)})\\
    &~ = 2\lips^2\norm{\vt[\jstate]-\avgconcat}^2\\
    &~~+ 4\nWorkers\lips (\obj(\vt[\avg[\state]])-\obj(\sol) - 
    \product{\vt[\avg[\state]]-\sol}{\grad\obj(\sol)})\\
    &~ = 2\lips^2\norm{\vt[\jstate]-\avgconcat}^2
    + 4\nWorkers\lips (\obj(\vt[\avg[\state]])-\obj(\sol)).
\end{align*}
\end{small}
In the last line we used that $\grad\obj(\sol)=0$.
Taking expectation gives the desired inequality.

\vskip 4pt
\textit{b)} The inequality is straightforward from \textit{a)} and the following decomposition
\begin{small}
\begin{equation}
    \notag
    \begin{aligned}
    &\norm{\grad\objMulti(\vt[\jstate])-\grad\objMulti(\vt[\jcstate])}^2\\
    &~\le 2\norm{\grad\objMulti(\vt[\jstate])-\grad\objMulti(\ones^\top \sol)}^2
    + 2\norm{\grad\objMulti(\ones^\top \sol)-\grad\objMulti(\vt[\jcstate])}^2.
    \end{aligned}
\end{equation}
\end{small}

\vskip -1pt
\textit{c)}
We first decompose
\begin{small}
\begin{equation}
    \notag
    \norm{\vt[\jvrgvec]}^2 = \norm{(\Id-\avgMat)\vt[\jvrgvec]+\avgMat\vt[\jvrgvec]}^2
    = \norm{(\Id-\avgMat)\vt[\jvrgvec]}^2 + \norm{\avgMat\vt[\jvrgvec]}^2.
\end{equation}
\end{small}
For the first term we simply use the definition of $\vt[\jvrgvec]$ to obtain 
\begin{small}
\begin{equation}
    \notag
    \begin{aligned}[b]
    &\norm{(\Id-\avgMat)\vt[\jvrgvec]}^2\\
    &~\le 2\norm{(\Id-\avgMat)\jgstate}^2 +2\norm{(\Id-\avgMat)(\grad\objMulti(\vt[\jstate])-\grad\objMulti(\vt[\jcstate]))}^2\\
    &~ \le 2\norm{\vt[\jgstate]-\avgconcat[\gstate]}^2 +
    2\norm{\grad\objMulti(\vt[\jstate])-\grad\objMulti(\vt[\jcstate])}^2.
    \end{aligned}
\end{equation}
\end{small}
From \cref{lem:yt-gt-sum}b, we know that
$\sumworker\vwt[\vrgvec]=\sumworker\grad\vw[\obj](\vwt[\state])$ or equivalently $\avgMat\vt[\jvrgvec]=\avgMat\grad\objMulti(\vt[\jstate])$. Thus for the second term we use again $\grad\obj(\sol)=0$ to get
\begin{small}
\begin{equation}
    \notag
    \begin{aligned}
    \norm{\avgMat\vt[\jvrgvec]}^2
    &= \nWorkers \Big\|\frac{1}{\nWorkers}\sumworker\nabla\vw[\obj](\vwt[\state]))\Big\|^2\\
    &= \nWorkers \Big\|\frac{1}{\nWorkers}\sumworker\grad\vw[\obj](\vwt[\state])- \frac{1}{\nWorkers}\sumworker\grad\vw[\obj](\sol) \Big\|^2\\
    &\le \sumworker \norm{\grad\vw[\obj](\vwt[\state])-\grad\vw[\obj](\sol)}^2.
    \end{aligned}
\end{equation}
\end{small}
Combining the above, taking expectation, and using \textit{a)} and \textit{b)} gives the desired result.
\end{proof}

\subsection{Proof of \cref{eq:lyapunov-contraction}}
\label{app:proof-lyap}

\begin{proof}
Let $\recmat_{\indg}$ denotes the $\indg$-th column of $\recmat$ and $\mathbf{e}_{\indg}$ denote the $\indg$-th canonical vector of $\R^4$.
First, $\vvec^{\top}\recmat_{1} = 1-\frac{\step\str\nWorkersActive}{2\nWorkers} = (1-\frac{\step\str\nWorkersActive}{2\nWorkers}) ~ \vvec^{\top}\mathbf{e}_1$.\\
Since $\step\le\frac{(1-\cfactor)^2}{14\lips}\sqrt{\frac{\nWorkers}{\nWorkersActive}}$, it holds
\begin{small}
\begin{equation}
    \notag
    \begin{aligned}
    \vvec^{\top}\recmat_{2} &= \frac{\sqrt{\nWorkersActive}(1-\cfactor)}{\nWorkers^{\frac{3}{2}}}
    \Bigg(
    \frac{\step\lips}{1-\cfactor}\sqrt{\frac{\nWorkersActive}{\nWorkers}}
    +\frac{10\step^2\lips^2}{1-\cfactor}\left(\frac{\nWorkersActive}{\nWorkers}\right)^{\frac{3}{2}}\\
    &\hspace{7em}+\frac{1+\cfactor}{2} + \frac{20\step^2\lips^2\nWorkersActive}{\nWorkers(1-\cfactor)}
    +\frac{\step\lips}{4(1-\cfactor)}\sqrt{\frac{\nWorkersActive}{\nWorkers}}
    \Bigg) \\
    &\le
    \frac{\sqrt{\nWorkersActive}(1-\cfactor)}{\nWorkers^{\frac{3}{2}}}
    \left(
    \frac{5\step\lips}{4(1-\cfactor)}\sqrt{\frac{\nWorkersActive}{\nWorkers}}
    +\frac{30\step^2\lips^2\nWorkersActive}{\nWorkers(1-\cfactor)} +\frac{1+\cfactor}{2}
    \right)\\
    &\le \frac{\sqrt{\nWorkersActive}(1-\cfactor)}{\nWorkers^{\frac{3}{2}}}
    \frac{3+\cfactor}{4} = \frac{3+\cfactor}{4} ~ \vvec^{\top}\mathbf{e}_2 .
    \end{aligned}
\end{equation}
\end{small}
With $\step\le\frac{(1-\cfactor)^2}{2304\lips}\left(\frac{\nWorkers}{\nWorkersActive}\right)^{\frac{3}{2}}$, we have
\begin{small}
\begin{equation}
    \notag
    \begin{aligned}
    \vvec^{\top}\recmat_{3} &=
    \frac{\step(1-\cfactor)}{96\nWorkers\lips}
    \Bigg(
    \frac{192\step\lips}{1-\cfactor}\left(\frac{\nWorkersActive}{\nWorkers}\right)^2\\
    &\hspace{6em}+\frac{384\step\lips}{1-\cfactor}\left(\frac{\nWorkersActive}{\nWorkers}\right)^{\frac{3}{2}}
    +\frac{1+\cfactor}{2}
    \Bigg) \\
    &\le
    \frac{\step(1-\cfactor)}{96\nWorkers\lips}
    \left(
    \frac{576\step\lips}{1-\cfactor}\left(\frac{\nWorkersActive}{\nWorkers}\right)^{\frac{3}{2}}
    +\frac{1+\cfactor}{2}
    \right)\\
    &\le \frac{\step(1-\cfactor)}{96\nWorkers\lips}
    \frac{3+\cfactor}{4} = \frac{3+\cfactor}{4} ~ \vvec^{\top}\mathbf{e}_3 .
    \end{aligned}
\end{equation}
\end{small}
Similarly, using $\step\le\frac{1}{576\lips}\sqrt{\frac{\nWorkers}{\nWorkersActive}}$, we get
\begin{small}
\begin{equation}
    \notag
    \begin{aligned}
    \vvec^{\top}\recmat_{4} &=
    \frac{\step}{12\nWorkers\lips}
    \Bigg(
    48\step\lips\left(\frac{\nWorkersActive}{\nWorkers}\right)^2
    +96\step\lips\left(\frac{\nWorkersActive}{\nWorkers}\right)^{\frac{3}{2}}\\
    &\hspace{5.5em}+\frac{\nWorkersActive}{2\nWorkers}
    +1-\frac{\nWorkersActive}{\nWorkers}
    \Bigg) \\
    &\le
    \frac{\step}{12\nWorkers\lips}
    \left(1-\frac{\nWorkersActive}{2\nWorkers}
    +144\step\lips\left(\frac{\nWorkersActive}{\nWorkers}\right)^{\frac{3}{2}}
    \right)\\
    &\le \frac{\step}{12\nWorkers\lips}
    \left(1-\frac{\nWorkersActive}{4\nWorkers}\right) = \left(1-\frac{\nWorkersActive}{4\nWorkers}\right) ~ \vvec^{\top}\mathbf{e}_4.
    \end{aligned}
\end{equation}
\end{small}
As for $\vt[\errgap]$, we note that $\step\le\frac{(1-\cfactor)^2}{2304\lips}\left(\frac{\nWorkers}{\nWorkersActive}\right)^{\frac{3}{2}}<\frac{1}{120\lips}\left(\frac{\nWorkers}{\nWorkersActive}\right)^{\frac{3}{2}}$ and thus
\begin{small}
\begin{equation}
    \notag
    \begin{aligned}
    \vvec^{\top}\recbias &=
    -\frac{\step\nWorkersActive}{\nWorkers}
    +\frac{20\step^2\lips\nWorkersActive^2}{\nWorkers^2}
    +40\step^2\lips\left(\frac{\nWorkersActive}{\nWorkers}\right)^{\frac{3}{2}}
    +\frac{\step\nWorkersActive}{6\nWorkers}
    +\frac{\step\nWorkersActive}{3\nWorkers} \\
    &\le
    -\frac{\step\nWorkersActive}{2\nWorkers}
    +60\step^2\lips\left(\frac{\nWorkersActive}{\nWorkers}\right)^{\frac{3}{2}}
    \le 0.
    \end{aligned}
\end{equation}
\end{small}
As $\step\le\frac{(1-\cfactor)^2}{14\lips}\sqrt{\frac{\nWorkers}{\nWorkersActive}}$ implies that $1-\frac{\step\str\nWorkersActive}{2\nWorkers}\ge\frac{3+\cfactor}{4}$, we have 
\begin{align*}
   \vvec^{\top}  \recmat  \leq  \cvgrate ~ \vvec^{\top} \Id
\end{align*}
where the inequality is elementwise and $\cvgrate=\max\left(1-\frac{\step\str\nWorkersActive}{2\nWorkers},1-\frac{\nWorkersActive}{4\nWorkers}\right)$. Since all the involved terms are non-negative, combining with the above inequalities gives
\begin{align*}
    \vvec^{\top}\update[\recvec] = \vvec^{\top}  \recmat ~ \vt[\recvec]+ \vt[\errgap] ~ \vvec^{\top}\recbias \le\cvgrate~\vvec^{\top}\vt[\recvec]
\end{align*}
which concludes the proof.
\end{proof}

\subsection{Proof of \cref{lem:Gt-bound}}
\label{app:proof-gt-bound}
\begin{proof}
\textit{a)}
From the definition of $\vt[\vrgvec]$ we can write
\begin{small}
\begin{equation}
    \notag
    \begin{aligned}
    &\norm{\vt[\jvrgvec]-\vt[\rightEigvec]\ones^{\top}\vt[\jvrgvec]}^2\\
    &~= \norm{\vt[\jgstate]-\vt[\rightEigvec]\ones^{\top}\vt[\jgstate]
    +(\Id-\vt[\rightEigvec]\ones^{\top})(\grad\objMulti(\vt[\jstate])-\grad\objMulti(\vt[\jcstate]))}^2\\
    &~\le 2\norm{\vt[\jgstate]-\vt[\rightEigvec]\ones^{\top}\vt[\jgstate]}^2
    + 2\norm{\Id-\vt[\rightEigvec]\ones^{\top}}^2\norm{\grad\objMulti(\vt[\jstate])-\grad\objMulti(\vt[\jcstate])}^2.
    \end{aligned}
\end{equation}
\end{small}
We conclude by using \[\norm{\Id-\vt[\rightEigvec]\ones^{\top}}^2\le2\norm{\Id}^2+2\norm{\vt[\rightEigvec]\ones^{\top}}^2\le2\nWorkers+2.\]
and taking expectation over the above inequalities.

\vskip 3pt
\textit{b)}
By Young's inequality,
\begin{small}
\begin{equation}
    \notag
    \norm{\vt[\jvrgvec]}^2
    \le 2\norm{\vt[\jvrgvec]-\vt[\rightEigvec]\ones^{\top}\vt[\jvrgvec]}^2
    + 2\norm{\vt[\rightEigvec]\ones^{\top}\vt[\jvrgvec]}^2.
\end{equation}
\end{small}
Using $\ones^{\top}\vt[\jvrgvec]=\ones^{\top}\grad\objMulti(\jstate)$, $\grad\obj(\sol)=0$, and the fact that $\vt[\rightEigvec]$ is a probability vector, we deduce that
\begin{small}
\begin{equation}
    \notag
    \begin{aligned}
    \norm{\vt[\rightEigvec]\ones^{\top}\vt[\jvrgvec]}^2
    &=\sumworker\left\|\vwt[\rightEigvec]\sumworker[\workeralt]\grad\vw[\obj][\workeralt](\vwt[\state][\workeralt])\right\|^2\\
    &\le \left\|\sumworker[\workeralt]\grad\vw[\obj][\workeralt](\vwt[\state][\workeralt])\right\|^2\\
    &= \left\|\sumworker\grad\vw[\obj](\vwt[\state])-\sumworker\grad\vw[\obj](\sol)\right\|^2\\
    &\le \nWorkers \sumworker \left\|\grad\vw[\obj](\vwt[\state])-\grad\vw[\obj](\sol)\right\|^2.
    \end{aligned}
\end{equation}
\end{small}
Combining the above two inequalities with \textit{a)} and \cref{lem:aux}a we get the desired result.
\end{proof}

\subsection{Proof of \cref{lem:effstep-control}}
\label{app:proof-effstep}

\begin{proof}
Using the independence assumption, we can write
\begin{equation}
    \notag
    \ex[\vt[\effstep]\vt[\srvp]]=\ex[\update[\leftEigvec^{\top}]\vt[\mixing]\vt[\sampMat]\vt[\rightEigvec]\vt[\srvp]]=\ex[\update[\leftEigvec^{\top}]]\ex[\vt[\mixing]\vt[\sampMat]]\ex[\vt[\rightEigvec]\vt[\srvp]].
\end{equation}

From \cref{asm:matrices}c we deduce that
\begin{equation}
    \notag
    \ex[\vt[\mixing]\vt[\sampMat]]
    \ge \ex[\diag(\diaglow\ones)\vt[\sampMat]]
    = \diaglow\diag(\bb{\sampprob}),
\end{equation}
where $\bb{\sampprob}=(\vw[\sampprob])_{\worker\in\workers}$.
We have shown in \eqref{eq:exp-ut} that $\ex[\update[\leftEigvec]]=\leftEigvecA$.
Notice that all the elements of $\leftEigvecA$ are positive according to the Perron-Frobenius theorem.
Thus, $\minv{\leftEigvecAEle}>0$ and we have
\begin{equation}
    \notag
    \ex[\update[\leftEigvec^{\top}]]\ex[\vt[\mixing]\vt[\sampMat]]\ge\leftEigvecA^\top\diaglow\diag(\bb{\sampprob})\ge\minv{\effstep}\ones^\top
\end{equation}
where \cref{asm:sampling} ensures that $\minv{\sampprob}>0$ and thus $\minv{\effstep}>0$.

On the other hand, $\vt[\leftEigvec]$ being a probability vector we can always upper bound $\vt[\leftEigvec]^\top\vt[\sampMat]$ by $\ones^\top$.
Using the non-negativity of $\vt[\rightEigvec]$ and $\vt[\srvp]$, we then obtain
\begin{equation}
    \notag
    \minv{\effstep}\ex[\ones^{\top}\vt[\rightEigvec]\vt[\srvp]]
    \le \ex[\vt[\effstep]\vt[\srvp]]
    \le \ex[\ones^{\top}\vt[\rightEigvec]\vt[\srvp]].
\end{equation}
This is exactly \eqref{eq:effstep-control} since $\vt[\rightEigvec]$ is a probability vector.
\end{proof}
\subsection{Additional simulations: Influence of $\cfactor$ and $\nWorkersActive$}
\label{apx:exp-setup}

In this appendix we illustrate via two examples how \ac{PPDS} could be influenced by different values of $\cfactor$ and $\nWorkersActive$.
It is however worth noticing that what we present here is specific to the communication strategies that we consider and we may observe different results for other communication strategies.

\subsubsection{Setups}
We base ourselves on the ridge regression experiment introduced in \cref{sec:num:datasets}.
As for the underlying random geometric graph we increase the radius to $0.3$ for better network connectivity.
We then consider the following two communication strategies (mixing matrices taken to be bi-stochastic so that \cref{thm:cvg-bistoch} applies).
\begin{enumerate}[\bfseries a), itemsep=1.5pt]
    \item \textit{Communication between active nodes and their neighbors:}
    Let $j>0$. At each round $\run$ we randomly select $j$ neighbors for each active node, and take the union of the active nodes and these selected neighbors as the communication nodes.
    The mixing matrix $\vt[\mixing]=\vt[\mixingalt]$ is set as the Metropolis matrix of the subgraph induced by the communication nodes.
    \item \textit{Communication between randomly selected nodes:}
    In this second setup we decouple communication from computation.
    In each round, independently of the sampling of the active nodes, we sample $5$ nodes and for each of these nodes we sample $1$ neighbor to form a group of at most $10$ communication nodes.
    The mixing matrix $\vt[\mixing]=\vt[\mixingalt]$ is again the Metropolis matrix of the subgraph induced by the communication nodes.
    Through Monte-Carlo estimation we get $\cfactor\approx0.99$.
\end{enumerate}
In the experiments, we fix $\nWorkersActive=10$ and take $j\in\{2,4,\ldots,18\}$ in setup \textbf{a)}, which gives $\cfactor$ ranging from $0.97$ to $0.87$.
As for setup \textbf{b)}, we choose $\nWorkersActive\in\{10, 15, \ldots, 50\}$.
As before, we do a grid search for the stepsize $\step$ and choose the optimal one within the range $\{10^{-k}, 2\leq k\leq 5 \}$.

\begin{figure}
    \centering
    \includegraphics[width=0.235\textwidth,valign=t]{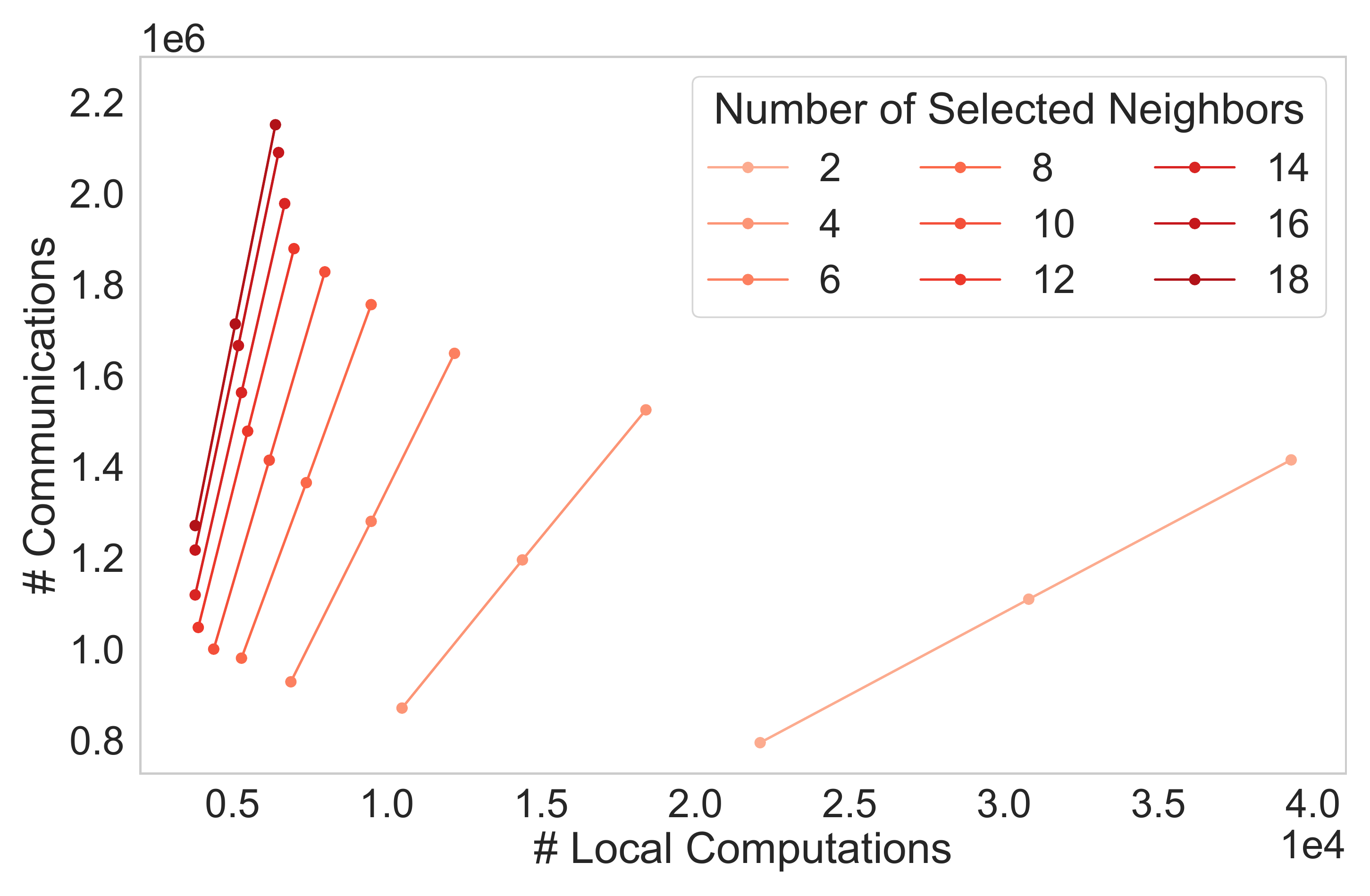}
    \includegraphics[width=0.24\textwidth,valign=t]{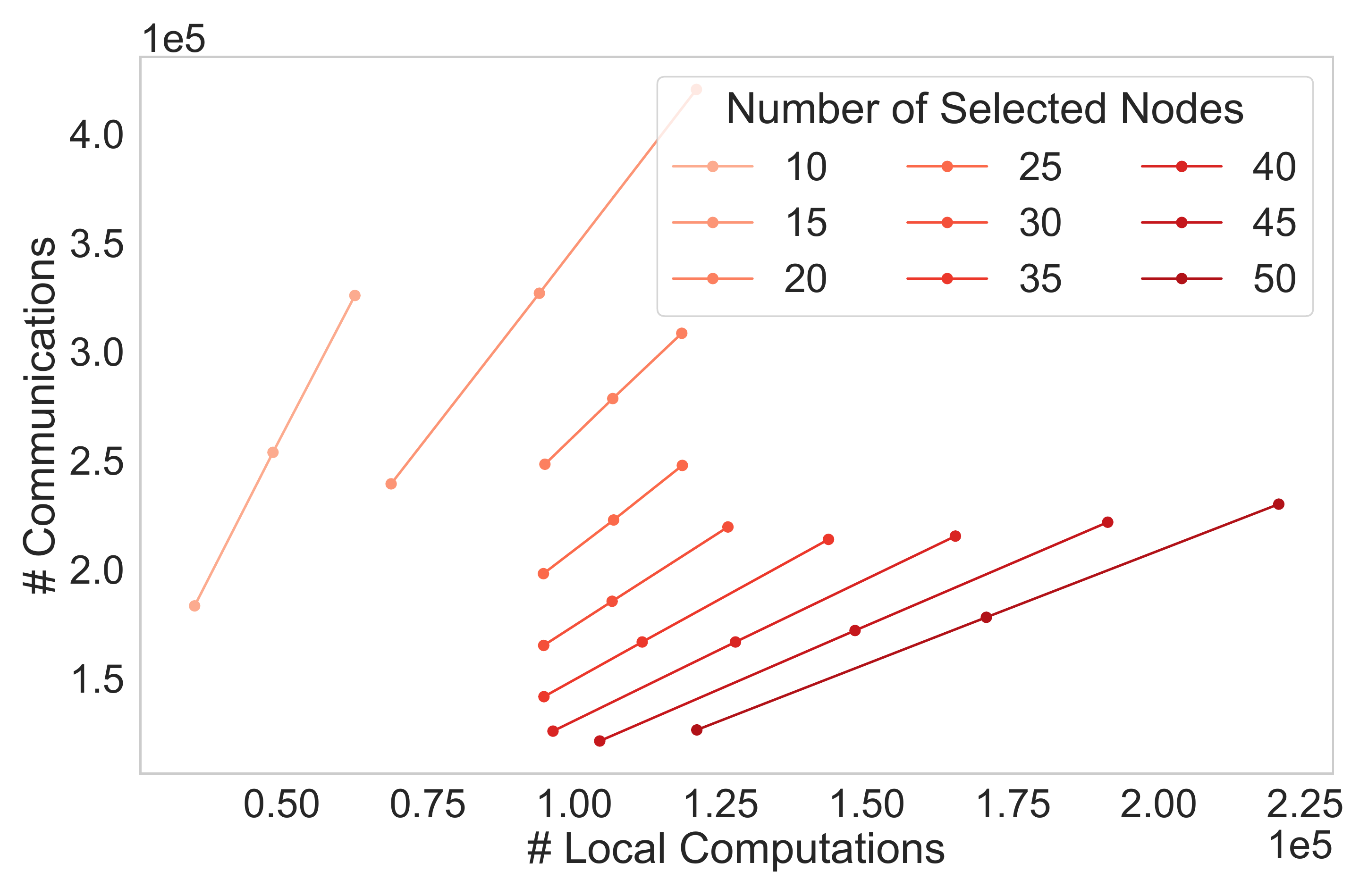}
    \caption{Illustration of the influence of $\cfactor$ and $\nWorkersActive$ on the performance of \acs{PPDS} for the ridge regression experiment.
    The left and the right figures are respectively for the communication strategies \textbf{a)} and \textbf{b)} described in \cref{apx:exp-setup}.
    We plot the number of local communications and computations that are needed for the algorithm to attain suboptimality values $10^{-2}$, $10^{-3}$, and $10^{-4}$ (from bottom left to top right). 
    Each line corresponds a specific configuration.}
    \label{fig:comm-comp-trade-off}
    \vspace{-1em}
\end{figure}

\subsection{Results}
In \cref{fig:comm-comp-trade-off} we plot the number of local communications and computations that are required for \ac{PPDS} to attain suboptimality values $10^{-2}$, $10^{-3}$, and $10^{-4}$ in different configurations.
For setup \textbf{a)} we observe a computation-communication trade-off: the more we communicate in each round, the less gradient computations but the more communications are needed to achieve a certain suboptimality value.
As for setup \textbf{b)} we observe two different behaviors depending on the stepsize.
At stepsize $\step=10^{-3}$ (the optimal stepsize for $\nWorkersActive\in\{10,15\}$) smaller the sample size $\nWorkersActive$ better the performance of the algorithm.
At stepsize $\step=10^{-4}$ (the optimal stepsize for $\nWorkersActive\in\{20,25,\ldots,50\}$) the best is to choose $\nWorkersActive\approx30$ for which communication and computation are well balanced and further increasing or decreasing $\nWorkersActive$ augments either computation or communication cost without really decreasing the other.


\vspace{-0.2em}
\section*{Acknowledgment}
This research was partially supported by
MIAI@Grenoble Alpes (ANR-19-P3IA-0003).

\ifCLASSOPTIONcaptionsoff
  \newpage
\fi



\bibliographystyle{IEEEtran}
\bibliography{references}
%



%

\vspace{-2.4em}

\begin{IEEEbiography}[{\vspace{-2em}\includegraphics[width=1in,height=1in,clip,keepaspectratio]{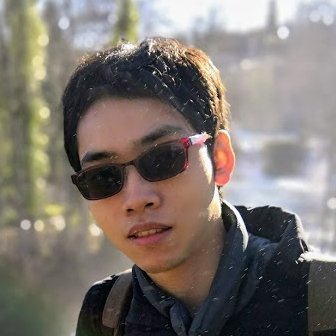}}]{Yu-Guan Hsieh}
  received a M.Sc. degree in Machine Learning and Applied Mathematics from ENS Paris Saclay in 2019 and a M.Sc. degree in Computer Science from ENS Paris in 2020. He is now pursuing his Ph.D. in Université Grenoble Alpes. His research interests include distributed optimization, online learning, and the study of learning-in-game dynamics.
\end{IEEEbiography}
\vspace{-4.8em}

\begin{IEEEbiography}[{\vspace{-2em}\includegraphics[width=1in,height=1in,clip,keepaspectratio]{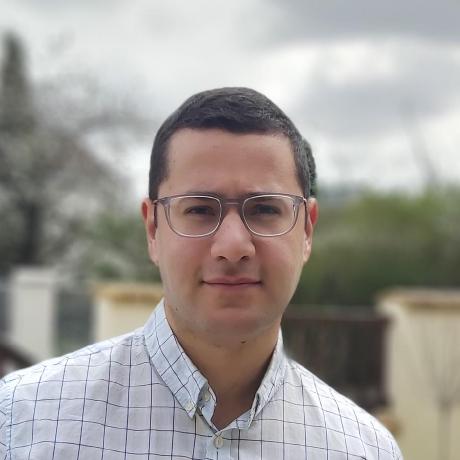}}]{Yassine Laguel}
  is a postdoctoral researcher at Rutgers University, working with Mert Gürbüzbalaban. Prior to this position, he received his Ph.D. from Univ. Grenoble Alpes, France, where he was advised by Jérôme Malick. His work focuses on the design and analysis of risk-averse algorithms in machine learning and stochastic programming, and their applications in both centralized, decentralized, or federated settings. 
\end{IEEEbiography}
\vspace{-4.4em}

\begin{IEEEbiography}[{\includegraphics[width=1in,height=1in,clip,keepaspectratio]{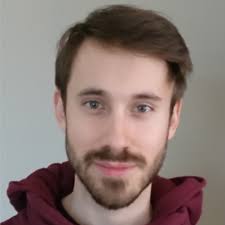}}]{Franck Iutzeler}
  was born in Besançon, France, in 1987. He received the Engineering degree from
  Telecom Paris, France, the M.Sc. degree from Sorbonne Université, Paris, France, both in 2010; and defended his Ph.D. in 2013 at Telecom Paris.
  In 2014--2015, he was a post-doctoral associate, first at Supélec (Gif-sur-Yvette, France) and then at Université Catholique de Louvain (Louvain-la-neuve, Belgium).
  Since 2015, he is an assistant professor at Univ. Grenoble Alpes.
  His research interests resolve around optimization methods and theory for data science, in particular focusing on machine learning problems, distributed optimization, robust optimization, and multi-agents systems.
\end{IEEEbiography}
\vspace{-3em}

\begin{IEEEbiography}[{\includegraphics[width=1in,height=1.25in,clip,keepaspectratio]{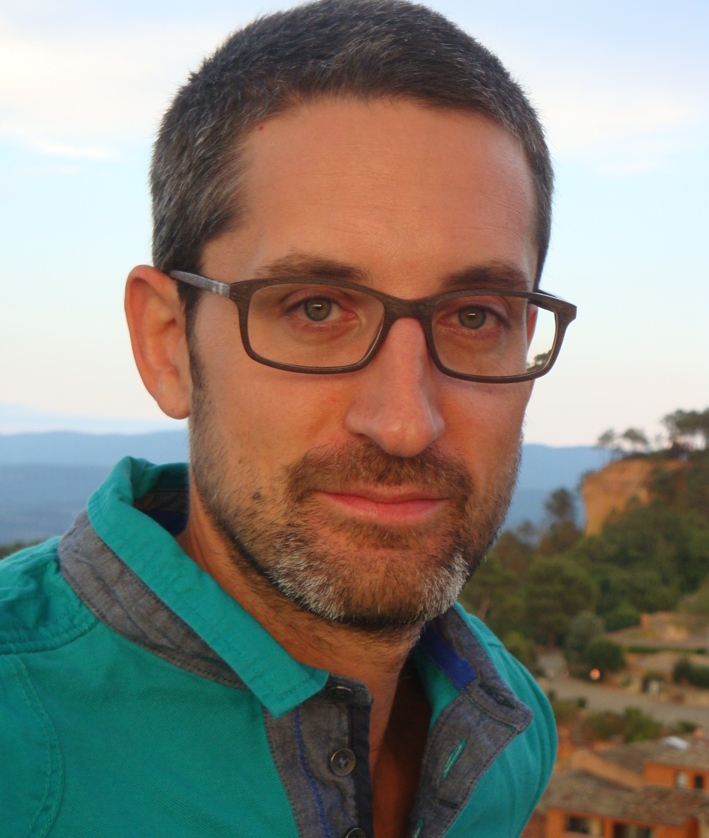}}]{Jérôme Malick}
  is a senior researcher at CNRS, leading since 2017 the DAO team of the Lab. Jean Kunztmann at University Grenoble Alpes.
  He completed his Ph.D. at Inria in 2005 and spent his post-doc at Cornell University in 2006. He received in 2009 the Robert Faure award for the most outstanding young researcher in Optim/OR in France. His research interests lie in mathematical optimization and its interactions, especially on distributed optimization, multi-agent learning, and optimization under uncertainty.
\end{IEEEbiography}








\end{document}